\documentclass[11pt, leqno]{amsart}

\usepackage{graphicx}
\usepackage{amssymb,amsmath,amsthm,amscd}
\usepackage[vcentermath,enableskew]{youngtab}

\usepackage{mathrsfs}
\usepackage{lscape}
\usepackage{enumerate}
\usepackage{upgreek}
\usepackage[usenames,dvipsnames]{color}
\usepackage[colorlinks=true, pdfstartview=FitV,
 linkcolor=blue,citecolor=blue,urlcolor=blue]{hyperref}
\usepackage{tikz}
\tikzset{dynkdot/.style={circle,draw,scale=.38}}
\usepackage{comment}
\usepackage{xspace}
\usepackage{marginnote}
\usepackage[normalem]{ulem}  
\usepackage[all]{xy}

\CompileMatrices

\theoremstyle{plain}
\newtheorem{theorem}{Theorem}[section]
\newtheorem{lem}[theorem]{Lemma}
\newtheorem{prop}[theorem]{Proposition}
\newtheorem{cor}[theorem]{Corollary}

\newtheorem{convention}[theorem]{Convention}

\theoremstyle{definition}
\newtheorem{definition}[theorem]{Definition}
\newtheorem{example}[theorem]{Example}
\newtheorem{remark}[theorem]{Remark}
\numberwithin{equation}{section} \numberwithin{figure}{section}
\numberwithin{table}{section}

\newcommand{\seteq}{\mathbin{:=}}

\newcommand{\Sp}{\operatorname{Sp}}
\newcommand{\USp}{\operatorname{USp}}

\newcommand{\U}{\operatorname{U}}
\newcommand{\SU}{\operatorname{SU}}

\newcommand{\ssqcup}{\mathop{\mbox{\normalsize$\bigsqcup$}}\limits}

\newenvironment{red}{\relax\color{red}}{\relax}
\newenvironment{blue}{\relax\color{blue}}{\hspace*{.5ex}\relax}
\newenvironment{jaune}{\relax\color{green}}{\hspace*{.5ex}\relax}

\newcommand{\ber}{\begin{red}}
\newcommand{\er}{\end{red}}
\newcommand{\beb}{\begin{blue}}
\newcommand{\eb}{\end{blue}}
\newcommand{\bjn}{\begin{jaune}}
\newcommand{\ejn}{\end{jaune}}

\newcommand{\Z}{\mathbb{Z}}

\newcommand{\C}{\mathbb{C}}

\newcommand{\m}{\mathfrak{m}}

\newcommand{\tyg}{\tiny\young}

\newcommand{\ep}{\epsilon}
\newcommand{\la}{\lambda}
\newcommand{\lap}{{\lambda'}}

\newcommand{\ee}{\end{enumerate}}
\newcommand{\ben}{\begin{enumerate}[{\rm (1)}]}
\newcommand{\bnum}{\begin{enumerate}[{\rm (i)}]}
 \newcommand{\bna}{\begin{enumerate}[{\rm (a)}]}
\newcommand{\bnA}{\begin{enumerate}[{\rm (A)}]}
\newcommand{\bc}{\begin{cases}}
\newcommand{\ec}{\end{cases}}

\begin{document}

\title[Auto-correlation functions for unitary groups]
{Auto-correlation functions for unitary groups}

\author[K.-H. Lee]{Kyu-Hwan Lee$^{\star}$}
\thanks{$^{\star}$This work was partially supported by a grant from the Simons Foundation (\#712100).}
\address{Department of
Mathematics, University of Connecticut, Storrs, CT 06269, U.S.A.}
\email{khlee@math.uconn.edu}

\author[S.-j. Oh]{Se-jin Oh$^{\dagger}$}
\address{Department of Mathematics, Ewha Womans University, Seoul 120-750, South Korea}
\email{sejin092@gmail.com}
\thanks{$^{\dagger}$The research of S.-j.\ Oh was supported by the Ministry of Education of the Republic of Korea and the National Research Foundation of Korea (NRF-2019R1A2C4069647)}

\date{\today}

\begin{abstract}
We compute the auto-correlations functions of order $m\ge 1$ for the characteristic polynomials of random matrices from certain subgroups of the unitary groups $\U(2)$ and $\U(3)$ by applying branching rules. These subgroups can be understood as analogs of Sato--Tate groups of $\USp(4)$ in our previous paper. This computation yields symmetric polynomial identities  with $m$-variables involving irreducible characters of $\U(m)$ for all $m \ge 1$ in an explicit, uniform way. 
\end{abstract}

\maketitle

\section{Introduction}

\subsection{Auto-correlation functions}  The distribution of characteristic polynomials of random matrices has been of great interest for their applications in mathematical physics and number theory.
Since Keating and Snaith \cite{KeaSn,KeaSn-1} computed averages of characteristic polynomials of random matrices in 2002 motivated in part by connections to number theory and in part by the importance of these averages in quantum chaos \cite{AS}, it has become clear that averages of characteristic polynomials are fundamental for random matrix models \cite{BDS,BS,BH,BH-2,FS,FS-1,MN}.

On the way of these developments, the auto-correlation functions of the distributions of
characteristic polynomials in the compact classical groups were computed by
Conrey, Farmer, Keating, Rubinstein and Snaith \cite{CFKRS,CFKRS-1} and by
Conrey, Farmer and Zirnbauer~\cite{CFZ,CFZ-1}. Later, Bump and Gamburd \cite{BG} obtained
different derivations of the formulas starting from (analogues of) the dual Cauchy identity and adopting a representation-theoretic method. Their results show that the auto-correlation functions are actually combinations of characters of classical groups.

For example, for the symplectic groups, an analogue of the dual Cauchy identity is due to Jimbo--Miwa \cite{JM} and Howe \cite{Howe}:
\begin{equation} \label{eqn-111} \prod_{i=1}^m \prod_{j=1}^g (x_i +x_i^{-1}+t_j+t_j^{-1}) = \sum_{\lambda \trianglelefteq (g^m)} \chi_\lambda^{\Sp(2m)}(x_1^{\pm 1}, \dots , x_m^{\pm 1})
\chi_{\tilde \lambda}^{\Sp(2g)}(t_1^{\pm 1}, \dots , t_g^{\pm 1}) ,\end{equation}
where $\chi^{\Sp(2m)}_{\lambda}$ is the irreducible character of $\Sp(2m, \mathbb C)$ associated with the partition $\lambda \trianglelefteq (g^m)$ and we set $\tilde \lambda=(m-\lambda'_g, \dots , m-\lambda'_1)$ with $\lambda' =(\lambda'_1, \dots , \lambda'_g)$ the transpose of $\lambda$. This identity can be considered as a reflection of Howe duality. Using this identity, Bump and Gamburd computed
the auto-correlation functions to obtain
\begin{equation} \label{eq-ssspp}
\int_{\USp(2g)}  \left ( \prod_{j=1}^m \det (I+ x_j \gamma) \right ) d \gamma = (x_1 \dots x_m)^g \, \chi^{\Sp(2m)}_{(g^m)} (x_1^{\pm 1}, \dots , x_m^{\pm 1}). \end{equation}

\subsection{Sato--Tate groups} The celebrated Sato--Tate conjecture for elliptic  curves (i.e. genus $1$ curves) predicts that the distribution of Euler factors of an elliptic curve is the same as the distribution of characteristic polynomials of random matrices from $\SU(2)$, $\U(1)$ or $N(\U(1))$, where $N(\U(1))$ is the normalizer of $\U(1)$ in $\SU(2)$. The conjecture is proven (under some conditions) by the works of R. Taylor, jointly with L. Clozel, M. Harris, and N. Shepherd-Barron \cite{CHT,T,HSBT}. For curves of higher genera,
 J.-P. Serre, N. Katz and P. Sarnak  \cite{Ser,KS} proposed a generalized Sato--Tate conjecture. Pursuing this
direction, K. S. Kedlaya and A. V. Sutherland \cite{KS09} and later
together with F. Fit{\' e} and V. Rotger \cite{FKRS}  made a list of
52 compact subgroups of $\USp(4)$ called {\em Sato--Tate groups}
that would classify all the distributions of Euler factors for
abelian surfaces. Recently, Fit{\'e}, Kedlaya and Sutherland showed that there are 410 Sato--Tate groups for abelian threefolds \cite{FKS-2}.

\subsection{Our previous work}
Inspired by the approach of  Bump and Gamburd, in a previous paper \cite{LO}, the authors computed the auto-correlation functions of characteristic polynomials for Sato--Tate groups $H \le \USp(4)$, which appear in the generalized Sato--Tate conjecture for genus $2$ curves. The result of \cite{LO} can be described as follows. Let $H \le \USp(4)$ be a Sato--Tate group. Then, for arbitrary $m \in \mathbb Z_{\ge 1}$, we have
\begin{align}
\int_{H}  \prod_{j=1}^m \det (I+ x_j \gamma)  d
\gamma & =  (x_1 \cdots x_m)^2 \sum_{b=0}^{m} \sum_{z=0}^{\lfloor \frac {m-b} 2 \rfloor} \mathfrak m_{(b+2z,b)} \chi_{(2^{m-b-2z}, 1^{2z})}^{\Sp(2m)} , \label{eqn-it}
\end{align}
where the coefficients $\mathfrak m_{(b+2z,b)}$ are the multiplicities of the trivial representation in the restrictions $\chi^{\Sp(4)}_{(b+2z,b)}\big |_H$ and are explicitly given in the paper \cite{LO} for all the Sato--Tate groups of abelian surfaces. Exploiting the representation-theoretic meaning of $\mathfrak m_{(b+2z,b)}$, the authors obtained this result by establishing branching rules for  $\chi^{\Sp(4)}_{(b+2z,b)}\big |_H$.

Moreover, since most of the Sato--Tate groups are disconnected, we can decompose the integral in \eqref{eqn-it} according to coset decompositions, and find that the characteristic polynomials over some cosets are \emph{independent} of the elements of the cosets. Combining this observation with the computations of branching rules, we obtain families of non-trivial identities of irreducible characters of $\Sp(2m, \mathbb C)$ for all $m \in \Z_{\ge1}$.
For example, we have, for any $m \in \mathbb Z_{\ge 1}$,
\begin{align*}
\tag{1}
\prod_{i=1}^m (x_i^2+x_i^{-2}) &= \sum_{b=0}^m \sum_{z=0}^{\lfloor \frac {m-b} 2 \rfloor} \psi_4(z,b) \chi_{(2^{m-b-2z}, 1^{2z})}^{\Sp(2m)} ,
\end{align*}
where $\psi_4(z,b)$ is defined on the congruence classes of $z$ and $b$ modulo $4$ by the table
\begin{center} {\scriptsize
\begin{tabular}{|c||c|c|c|c|}
\hline
$z\backslash b$ & 0&1&2&3\\ \hline \hline
0&$1$ & $-1$&$0$& $0$ \\ \hline
1&$0$&$1$ & $-1$ & $0$ \\ \hline
2& $-1$ & $1$ & $0$&$0$ \\ \hline
3&$0$ & $-1$ & $1$ & $0$ \\ \hline
\end{tabular}}\end{center}
Notice that the irreducible characters $\chi_\lambda^{\Sp(2m)}$ are
symmetric functions with the number of terms growing very fast as
$m$ increases, but that the coefficients $\psi_4(z,b)$ are
independent of $m$. In order to produce the left-hand side of the
identities, there must be systematic cancelations in the right-hand
side. 

\subsection{Schur functions}
The Schur functions $S_\la$ form  the \emph{distinguished}
self-dual basis of the ring of symmetric functions.  They appear
naturally in representation theory, algebraic combinatorics,
enumerative combinatorics, algebraic geometry and quantum physics.
In particular, (i) every Schur function corresponds to an
irreducible character of the unitary group, which implies the ring of symmetric functions
form the Grothendieck ring for unitary groups, (ii) it has various
combinatorial realizations in  various  aspects of algebraic
combinatorics. Hence understanding the properties of Schur
functions takes a center stage in these research areas. One of the key
features of understanding Schur functions is how other symmetric
functions can be expressed in the basis of Schur functions; that is,
computing the coefficients of $S_\la$, called \emph{Schur
coefficients}, in the expansions. One of the well-known instances is
the (inverse of) \emph{Kostka matrix}, which can be understood as
Schur coefficients for monomial (complete) symmetric functions  in
algebraic combinatorics, and as \emph{composition
multiplicities} of $V(\la)$ in the permutation
representation $W(\la)$ in representation theory. Recall that
the (inverse of) Kostka matrix is  a  uni-upper triangular matrix with
integer coefficients and entries in the Kostka matrix has a description
in terms of semistandard Young tableaux. However, the closed-form
formulas for entries in the (inverse of) Kostka matrix are not available  in general. Another well-known
instance is \emph{Littlewood--Richardson rule}, which can be
understood as  Schur coefficients for a product of two
Schur functions in algebraic combinatorics, and as
composition multiplicities of $V(\la)$'s in the
 tensor product  $V(\mu)\otimes V(\eta)$ in representation theory.

\subsection{Main Result}

In what follows, we describe the main result of this paper and its application.

\smallskip 

\noindent
({\bf M}) We compute explicitly the auto-correlation functions for  $H=\U(1) \le \U(g)$ $(g=2,3)$ and for the subgroups $H \le \U(g)$ $(g=2,3)$ defined in~\eqref{eq: A Sato-Tate} below. Namely, for any $m \in \Z_{\ge 1}$, we obtain 
$$\int_{H\le \U(g)} \prod_{i=1}^m \det (I+ x_i \gamma)  \, d \gamma  = \sum_{\lambda \trianglelefteq (g^m)} \mathfrak m_{\lambda'}(H)\, S^{\U(m)}_\lambda(\mathbf x),$$
where the coefficients $\mathfrak m_{\lambda'}(H)$ are completely determined. Here $S^{\U(m)}_{\la}(\mathbf x)$ denotes the character of the irreducible  representation $V(\la)$ of the unitary group $\U(m)$, which are Schur functions. 

\smallskip 

\noindent
({\bf A}) As an application of the main result, we give closed-form formulas of Schur coefficients $\mathfrak{c}_{g_1,g_2}^{\la}$ for special infinite families of symmetric functions $\mathfrak{t}^{(m)}_{g_1,g_2}(\mathbf{x})$, which also have simple expansions in terms of monomial symmetric functions
$\{ \mathsf{m}^{(m)}_\la(\mathbf x) \}$. That is, for \emph{any} $m \in \Z_{\ge 1}$ and $1\le g_1+g_2 \le 3$, we obtain 
\begin{align} \label{eq: identities}
\mathfrak{t}^{(m)}_{g_1,g_2}(\mathbf{x}) \seteq  \prod_{i=1}^{m}  (1\pm x_i^{g_1})(1 \pm x_i^{g_2}) = \sum_{\la \trianglelefteq ( (g_1+g_2)^m) } \mathfrak{c}_{g_1,g_2}^{\la}S^{\U(m)}_{\la} (\mathbf x),
\end{align}
where the Schur coefficients $\mathfrak{c}_{g_1,g_2}^{\la}$ are given in closed-form formulas.
Thus $\mathfrak{c}_{g_1,g_2}^{\la}$ can be understood as a simple 
combination of entries in the inverse of Kostka matrix as
$\mathfrak{t}^{(m)}_{g_1,g_2}(\mathbf{x})$ has an expansion in $ \mathsf{m}^{(m)}_\la(\mathbf x)$ with coefficients from $\{ 1,0,-1 \}$.

\medskip 

Let us explain the main result and its application in more detail.  To adopt the same strategy as in $\USp(m)$, we
first introduce several subgroups $H_{g}$, $H'_{g,4}$ and $H_{g,4}$
of $\U(g)$ $(g=2,3)$ as follows, which play the role of Sato--Tate groups in
\cite{LO}:
\begin{equation}\label{eq: A Sato-Tate}
\begin{aligned}
H_2 &:= \langle  \U(1),J_2 \rangle,  \quad    H'_{2,4} := \langle \U(1), \pmb \zeta_{2,4} \rangle \le   H_{2,4}:= \langle \U(1), J_2, \pmb \zeta_{2,4} \rangle \le \U(2), \\
H_3 &: = \langle \U(1), J_3 \rangle,   \quad  H'_{3,4}:= \langle \U(1),  \pmb \zeta_{3,4} \rangle \le   H_{3,4}:= \langle \U(1), J_3, \pmb \zeta_{3,4} \rangle \le \U(3), \end{aligned}
\end{equation}
where
$$
J_2 \seteq \left( \begin{matrix} 0 & 1 \\ -1 & 0  \end{matrix}\right), \ \
\pmb \zeta_{2,4} := \left( \begin{matrix}
\sqrt{-1} & 0  \\
0 & \sqrt{-1}
\end{matrix}\right), \ \
J_3: = \left(
\begin{matrix}
0 & 1 & 0 \\ -1 & 0 & 0 \\  0& 0 & 1
\end{matrix}
\right), \ \
\pmb \zeta_{3,4} := \left( \begin{matrix}
 \sqrt{-1} & 0 & 0 \\
0 & \sqrt{-1} & 0 \\
0 & 0 & 1
\end{matrix}\right).
$$

Considering these subgroups in $\U(g)$ $(g=2,3)$, we present the main result more precisely.

\smallskip

\noindent \textbf{Main Theorem.} Let $H \le \U(g)$ $(g=2,3)$ be  $\U(1)$ or a group  in~\eqref{eq: A Sato-Tate}. Then, for any $m \in \Z_{\ge 1}$, we have
$$\int_{H\le \U(g)} \prod_{i=1}^m \det (I+ x_i \gamma)  \, d \gamma  = \sum_{\lambda \trianglelefteq (g^m)} \mathfrak m_{\lambda'}(H)\, S^{\U(m)}_\lambda(\mathbf x)$$
where the coefficient $ \mathfrak{m}_{\lambda'}(H)$ are the \emph{composition multiplicities} of
the trivial representation in the restriction
$\chi^{\U(g)}_{\la'}|_H$ and are explicitly given in {\rm
Theorem~\ref{thm :g21}, Theorem~\ref{thm:g24}, Theorem~\ref{thm
:g3}} and {\rm Theorem~\ref{thm :g34}.}

\smallskip

This theorem can be interpreted as a result on branching rules from
$\U(g)$ to $H$ in representation theory. To prove the above theorem, we analyze the representation structure of $V(\la')$ over $\U(g)$ with respect to $H$
and determine a certain list of linearly independent subsets in
$V(\la')$.

\smallskip

For an application of \textbf{Main Theorem}, we observe (i) $H$'s are decomposed into disconnected cosets
\begin{align*}
H_2&= \U(1) \sqcup \underline{J_2 \U(1)}, & H_{2,4} &= H'_{2,4} \sqcup \underline{J_2\U(1)}  \sqcup \underline{\pmb \zeta_{2,4} J \U(1)} , \\
H_3 & = \U(1) \sqcup \underline{J_3\U(1)}, & H_{3,4} &= H'_{3,4} \sqcup \underline{J_3\U(1)}  \sqcup \underline{\pmb \zeta_{3,4} J \U(1)} ,
\end{align*}
and (ii)
the characteristic polynomials in the underlined cosets of $H$ are independent of elements of the coset.  These observations enable us to obtain closed-form identities involving Schur functions.

\smallskip

\noindent \textbf{Application.} For arbitrary $m \in \Z_{\ge 1}$, we have the following
identities$\colon$
\bnA
\item  \label{it: mainA} $\displaystyle\prod_{i=1}^m (1+x_i^2) = \sum_{b=0}^m \sum_{j=0}^{\lfloor \frac {m-b} 2 \rfloor} (-1)^j \ S^{\U(m)}_{(2^b, 1^{2j})}(\mathbf x)$ $\phantom{LLLLLLLLLL} ($derived from $H_2$,   {\rm Theorem~\ref{thm :g21}}$)$,
\item  \label{it: mainB} $\dfrac{1}{2} \left( \displaystyle\prod_{i=1}^m (1+x_i^2)+ \prod_{i=1}^m (1-x_i^2) \right) = \displaystyle\sum_{\substack{  (b+2j,b) \trianglelefteq (m^2)    \\   b+j \equiv_2 0 }} (-1)^{j} S^{\U(m)}_{(2^{b}, 1^{2j})}(\mathbf x)$
$ \\ \phantom{LLLLLLLLLLLLLLLLLLLLLLLLLLLLL}($derived from $H_{2,4}$ and $H'_{2,4}$,   {\rm Theorem~\ref{thm:g24}}$)$,
\item  \label{it: mainC} $\displaystyle \prod_{i=1}^m (1+x_i)(1+x_i^2)  =\sum_{\substack{\la \trianglelefteq (3^m) \\ \la=(3^k,2^b,1^z)}} \tau(z,b)\,  S^{\U(m)}_{\la}(\mathbf x)$ $ \phantom{LLLLL} ($derived from $H_3$,   {\rm Theorem~\ref{thm :g3}}$)$,
\item  \label{it: 4} $\displaystyle \frac 1 2 \left( \prod_{i=1}^m (1+x_i)(1+x_i^2)+  \prod_{i=1}^m (1+x_i)(1-x_i^2) \right)  =\sum_{\la \trianglelefteq (3^m)} \omega_\epsilon(z,b') \,  S^{\U(m)}_{\la}(\mathbf x),$
$ \\ \phantom{LLLLLLLLLLLLLLLLLLLLLLLLLLLLL}($derived from $H_{3,4}$ and $H'_{3,4}$,   {\rm Theorem~\ref{thm :g34}}$)$,
\end{enumerate}
where
\begin{enumerate}
\item[{\rm (i)}] $\tau(z,b) \in \{0,1\}$ is defined on the congruence of $z$ and $b$ modulo $4$ as in the following table
\begin{center} $\tau(z,b)=$ {\scriptsize
\begin{tabular}{|c||c|c|c|c|}
\hline
$z\backslash b$ & 0&1&2&3\\ \hline \hline
0&$1$ & $1$&$0$& $0$ \\ \hline
1&$1$&$0$ & $-1$ & $0$ \\ \hline
2&$0$&$-1$ & $-1$ & $0$ \\ \hline
3&$0$&$0$ & $0$ & $0$ \\ \hline
\end{tabular}}
 \end{center}
\item[{\rm (ii)}] $\omega_\epsilon(z,b')$ depends on the congruence of $z$ and $b$ modulo $2,4$ and given in {\rm Corollary~\ref{cor-fi}}.
\end{enumerate}

\smallskip

Combining the identities  in \textbf{Application} and replacing $x_i$ with $-x_i$, we obtain all the other identities in~\eqref{eq: identities} (see ~\eqref{alP}, Corollary~\ref{cor:fep}, Corollary~\ref{cor:fep2} and Remark~\ref{rem:fep3}). 

\medskip

Note that the identities obtained involves \emph{negative} Schur coefficients and the number of
terms in Schur functions grows enormously as $m$ increases. However, our result implies that such combinations of
Schur functions  have miraculous cancellations and yield symmetric functions with \emph{positive} coefficients.
Furthermore, the identities state that the Schur coefficients  do \emph{not} depend on $m$ (see Example~\ref{ex: g2 huge}, Example~\ref{ex: g3 huge} and Example~\ref{ex: g34 huge}). These identities seem intriguing from the viewpoint of representation
theory and algebraic combinatorics. It might have been difficult for us to expect that such identities exist, without regard to the auto-correlation functions and branching rules of the newly introduced groups in \eqref{eq: A Sato-Tate}.

\subsection{Organization of the paper}
In Section~\ref{sec: Preliminaries}, we review the necessary
backgrounds for auto-correlation functions  and the dual Cauchy identity.
In Section~\ref{sec: g=2}, we compute the auto-correlation functions of $H$'s of $\U(2)$ and establish the corresponding identities involving Schur functions. In Section~\ref{sec: g=3}, we present the auto-correlation functions of $H$'s of $\U(3)$ and consider the corresponding identities
 by analyzing the representation structure
of $V(\la)$ with respect to $H$'s. But we postpone a part of the proof to Section~\ref{sec:
Cardinality}, which is devoted to determine the composition multiplicities of trivial representations in
$\chi^{\U(g)}_{\la'}|_H$'s. This amounts to the proof for $\U(3)$. We
convert this problem into  counting the pairs of integers encoding certain information from representation theory.
 By expressing the cardinalities as closed-form
formulas, we complete the proof.

\begin{convention}
Throughout this paper, we keep the following conventions.
\begin{enumerate}
\item[{\rm (i)}] For a statement $P$, the notation $\delta(P)$ is equal to $1$ or $0$ according to whether $P$ is true or not.
\item[{\rm (ii)}]  For $m,m' \in \Z$ and $k \in \Z_{>0}$, we write $m\equiv_k m'$ if $k$ divides $m-m'$, and $m \not\equiv_k m'$ otherwise.
\end{enumerate}
\end{convention}


\section{Dual Cauchy identity and auto-correlation functions} \label{sec: Preliminaries}

In this section, we fix notations and review the dual Cauchy identity and establish a general formula for the auto-correlation functions of characteristic polynomials.

\medskip

A {\em partition} $\la=(\la_1 \ge \cdots \ge \la_k)$ is a non-increasing sequence of non-negative integers $\la_i$. Define $|\lambda|=\sum_{i=1}^k \lambda_i$ and $\ell(\lambda) =k$. We write $\la=(m^k)$ when $m=\la_1=\cdots=\la_k$. More generally, a partition is written as $(m_1^{k_1}, m_2^{k_2}, \dots , m_s^{k_s})$ for $m_1>m_2> \cdots > m_s$ and $k_i \ge 1$. For two partitions $\la= (\la_1 \ge \cdots \ge \la_k)$ and $\mu=(\mu_1 \ge \cdots \ge  \mu_l )$, we define a partial order $\la \trianglelefteq \mu$ if $k \le l$ and $\la_i \le \mu_i$ for all $i=1,2,\dots , k$. A partition $\la$ corresponds to a Young diagram, and the {\em transpose} $\la'$ is defined to be the partition corresponding to the transpose of the Young diagram of $\lambda$.

Let $\U(g)$ be the unitary group for $g \ge 1$. For a partition $\la$ with at most $g$ parts, let $S^{\U(g)}_\la$ be the Schur function associated with $\la$. It is well-known that $S^{\U(g)}_\lambda$ is the irreducible character of $\U(g)$ with highest weight $\la$. Denote by $V_g(\la)$ the representation space of $S_\lambda^{\U(g)}$. When $g$ is clear from the context, we will simply write $V(\la)$.

\begin{definition} \label{def-m-H}
Let $H$ be a closed subgroup of $\U(g)$. Define $\mathfrak m_{\lambda}(H)$ to be the multiplicity of the trivial representation $1_H$ in the restriction of $V(\la)$ to $H$.
\end{definition}

\subsection{Dual Cauchy identity} Let us recall the dual Cauchy identity (see, e.g.,  ~\cite[(9)]{BG}) :
\begin{lem} \label{lem-dual-Cauchy} For any $m \ge 1$ and $g \ge 1$, we have
\begin{align*}
\prod_{i=1}^m\prod_{j=1}^g (1+x_it_j) = \sum_{\lambda \trianglelefteq  (g^m)} S^{\U(m)}_{\la}(\mathbf x)S^{\U(g)}_{\lap}(\mathbf t),
\end{align*}
where $\mathbf x=(x_1, \dots , x_m)$ and $\mathbf t=(t_1, \dots, t_g)$.
We also have
\begin{align*}
\prod_{i=1}^m\prod_{j=1}^g (1-x_i t_j) = \sum_{\lambda \trianglelefteq  (g^m)} (-1)^{|\la|} S^{\U(m)}_{\la}(\mathbf x)S^{\U(g)}_{\lap}(\mathbf t),
\end{align*}
by replacing $x_i$ to $-x_i$.
\end{lem}

Following Proposition 3.3 in \cite{LO}, we obtain a formula for the auto-correlation functions.

\begin{prop} \label{prop-main}
Let $H$ be a subgroup of $\U(g)$ and $d\gamma$ be the probability Haar measure on $H$. Then, for each $m\ge 1$, the auto-correlation function for the distribution of characteristic polynomials of $H$ is given by
\begin{align} \label{main-f}
\int_{H} \prod_{i=1}^m \det (I+ x_i \gamma)  \, d \gamma & = \sum_{\lambda \trianglelefteq (g^m)} \mathfrak m_{\lambda'}(H)\, S^{\U(m)}_\lambda(\mathbf x), \\
\int_{H} \prod_{i=1}^m \det (I- x_i \gamma)  \, d \gamma &= \sum_{\lambda \trianglelefteq (g^m)} (-1)^{|\la|} \mathfrak m_{\lambda'}(H)\, S^{\U(m)}_\lambda(\mathbf x). \label{main-fn}
\end{align}
\end{prop}
\begin{proof}
Let $t_1, \dots , t_g$ be the eigenvalues of $\gamma \in H$.
Since we have
\[ \det (I + x_i \gamma) = \prod_{j=1}^g (1+x_i t_j),\]
it follows from Lemma \ref{lem-dual-Cauchy} that
\begin{align*} 
& \int_{H}  \prod_{j=1}^m \det (I+ x_j \gamma)  d \gamma = \int_{H}  \prod_{i=1}^m \prod_{j=1}^g (1+x_i t_j) d \gamma  \allowdisplaybreaks \\
&= \int_{H}  \sum_{\lambda \trianglelefteq  (g^m)} S^{\U(m)}_{\la}(\mathbf x)S^{\U(g)}_{\lap}(\mathbf t) d\gamma =  \sum_{\lambda \trianglelefteq  (g^m)} S^{\U(m)}_{\la}(\mathbf x) \int_{H}S^{\U(g)}_{\lap}(\mathbf t) d\gamma .   \nonumber
\end{align*}
From Schur orthogonality (for example, \cite{Bu}), the integral $\int_{H}
S^{\U(g)}_{\lap}(\mathbf t)  d\gamma$
is equal to the multiplicity of the trivial representation $1_H$ of $H$ in the restriction of $S^{\U(g)}_{\lap}$ to $H$, which is  equal to $\mathfrak m_{\lambda'}(H)$ by Definition \ref{def-m-H}. This establishes the first identity.
The second identity follows from replacing $x_i$ with $-x_i$ in the first identity.
\end{proof}

\begin{cor}
Suppose that $-I \in H$. Then $\mathfrak m_\lap(H) =0$ whenever $|\lap|$ is odd.
\end{cor}

\begin{proof}
If $-I \in H$, we have
\[  \int_{H}  \prod_{j=1}^m \det (I+ x_j \gamma)  d \gamma = \int_{H}  \prod_{j=1}^m \det (I- x_j \gamma)  d \gamma .\]
Our assertion follows from \eqref{main-f} and \eqref{main-fn} by comparing the right-hand sides, since the Schur functions are linearly independent.
\end{proof}

We recall the classical branching rule from $\U(g)$ to $\U(g-1)$ for $g \ge 2$.
\begin{prop}   \label{thm: branching A2 to A1}
Let $V_g(\lambda)$ be the irreducible representation of $\U(g)$ with highest weight $\lambda$. Then we have
$$   [V_g(\la):V_{g-1}(\mu)] \le 1$$ for any partition $\mu$ with at most $g-1$ parts.
Furthermore, $   [V_{g}(\la):V_{g-1}(\mu)] = 1$ precisely when
$$   \la_1 \ge  \mu_1 \ge \la_2 \ge \mu_2 \ge \cdots \ge \mu_{g-1} \ge \la_g.$$
\end{prop}

\section{Identities for $g=2$} \label{sec: g=2}
In this section, we consider some disconnected subgroups $H$ of $\U(2)$ and compute $\mathfrak m_{\lap}(H)$ for $\la \trianglelefteq (2^m)$ in \eqref{main-f}. This computation produces identities involving Schur functions $S_{\lambda}^{\U(m)}$ for all $m \in \mathbb Z_{\ge 1}$.

\medskip

We identify  $\U(1)$ with the subgroup $\left \{ \begin{pmatrix} t & 0 \\ 0 & t^{-1}  \end{pmatrix} : t\in \C , |t|=1 \right \} \le \U(2)$.

\subsection{Subgroup $\langle \U(1), J\rangle$}
Let us consider the subgroup $H_2$ of $\U(2)$ generated by
$\U(1)$ and
$J \seteq \left( \begin{matrix} 0 & 1 \\ -1 & 0  \end{matrix}\right)$,
i.e. $$H_2 = \langle  \U(1),J \rangle \le \U(2).$$
Then we have \begin{equation} \label{h2j} H_2= \U(1) \sqcup J\U(1) .\end{equation}

\begin{theorem} \label{thm :g21}For any partition $(a,b) \trianglelefteq (m^2)$, we have
\[  \m_{(a,b)}(\U(1))= \delta ( a \equiv_2 b) \quad \text{ and } \quad \mathfrak m_{(a,b)}(H_2)=  \delta( a \equiv_4 b). \]
Furthermore, for any $m \in \Z_{\ge1}$, we have
\begin{align}  \label{eq: g2 general}
\prod_{i=1}^m (1+x_i^2)= \sum_{\substack{ \la \trianglelefteq (2^m) \\ \la'=(b+2j,b)}} (-1)^j \ S^{\U(m)}_{\lambda}(\mathbf x) = \sum_{b=0}^m \sum_{j=0}^{\lfloor \frac {m-b} 2 \rfloor} (-1)^j \ S^{\U(m)}_{(2^b, 1^{2j})}(\mathbf x) ,
\end{align} where we set $j:=(a-b)/2$.
\end{theorem}

\begin{proof}
 For any $\gamma = \begin{pmatrix} t &0\\0&t^{-1} \end{pmatrix} \in\U(1)$, we have
 $\det(I+xJ\gamma) = 1+x^2$.
Let $du=du(\gamma)$ be the probability Haar measure on $\U(1) \le \U(2)$. By Proposition \ref{prop-main} and \eqref{h2j}, we have
\begin{align}
\sum_{\la \trianglelefteq (2^m)} \m_{\lap}(H_2) S^{\U(m)}_{\la}(\mathbf x)&=  \int_{H_2}  \prod_{i=1}^m \det(I+x_i\gamma) d\gamma \nonumber \\ & = \dfrac{1}{2} \int_{\U(1)} \prod_{i=1}^m \det(I+x_i\gamma) du
+  \dfrac{1}{2}  \int_{\U(1)} \prod_{i=1}^m \det(I+x_iJ\gamma)du \nonumber \\
& = \dfrac{1}{2}  \sum_{\la \trianglelefteq (2^m)} \m_{\lap}(\U(1)) S^{\U(m)}_{\la}(\mathbf x)
+  \dfrac{1}{2}  \int_{\U(1)} \prod_{i=1}^m (1+x_i^2) du \nonumber \\
&  = \dfrac{1}{2}   \sum_{\la \trianglelefteq (2^m)} \m_{\lap}(\U(1)) S^{\U(m)}_{\la}(\mathbf x)
+  \dfrac{1}{2} \prod_{i=1}^m (1+x_i^2)  .  \label{eqn-pro}
\end{align}

Let $v_1=(1,0)$ and $v_2=(0,1)$ be the standard unit vectors of $V:=\mathbb C^2$, and consider the standard representation of $\U(2)$ on $V$, and let $\mathbf{det}$ be the one-dimensional representation of $\U(2)$ defined by the determinant. For $\lap=(a,b) \trianglelefteq (m^2)$, we have
$V(\lambda') \cong \mathbf{det}^b \otimes \mathrm{Sym}^{a-b}(V)$.
Thus the trivial $\U(1)$-module is generated by $v_1^j v_2^j$ only when $a-b$ is even, where we set $j:=(a-b)/2$. In other word, we have
\[ \m_{(a,b)}(\U(1))= \delta ( a \equiv_2 b) .\]
Furthermore, since $J$ sends
$v_1\mapsto -v_2$ and $v_2 \mapsto v_1$, we see that $v_1^j v_2^j$ is fixed by $J$ when $j$ is even. Therefore,
\[\m_{(a,b)}(H_2)=  \delta( j \equiv_2 0)= \delta ( a \equiv_4 b) . \]

Note that, when $\lap=(b+2j,b)$, we have $\la=(2^b, 1^{2j})$.
Now it follows from \eqref{eqn-pro} that
\[
\prod_{i=1}^m (1+x_i^2) =  \sum_{b=0}^m \sum_{j=0}^{\lfloor \frac {m-b} 2 \rfloor}\left ( 2 \delta(j \equiv_2 0) -1 \right ) \ S^{\U(m)}_{(2^b, 1^{2j})}(\mathbf x)=  \sum_{b=0}^m \sum_{j=0}^{\lfloor \frac {m-b} 2 \rfloor} (-1)^j \ S^{\U(m)}_{(2^b, 1^{2j})}(\mathbf x). \qedhere
\]
\end{proof}

\begin{remark} \label{rmk: g2}
The left hand side of~\eqref{eq: g2 general} can be written as a simple combination of the monomial symmetric functions. Namely, we have
$$ \prod_{i=1}^m (1+x_i^2)= \sum_{k=0}^m \mathsf{m}^{(m)}_{(2^k)} (\mathbf x) ,$$
where $\mathsf{m}^{(m)}_{\la}$ denote the monomial symmetric functions in $m$-variables associated with partitions $\la$ with $\ell(\la) \le m$.
\end{remark}

\begin{example} \label{ex: g2 huge}
Let us see an example for the case $m = 10$ in Theorem~\ref{thm :g21}. We have
 \begin{align}\label{eq: sl2 ex10}
\prod_{i=1}^{10} (1+x_i^2) =  \sum_{b=0}^{10} \sum_{j=0}^{\lfloor \frac {10-b} 2 \rfloor} (-1)^j \ S^{\U(10)}_{(2^b, 1^{2j})}(\mathbf x).
\end{align}
Note that $S^{\U(10)}_{(2^4, 1^2)}(\mathbf x)$ appears in the right
hand side of~\eqref{eq: sl2 ex10} with the coefficient $-1$ since $j
= 1$. As a polynomial itself, $S^{\U(10)}_{(2^4, 1^2)}(\mathbf x)$
contains $8701$ monomial terms and $S^{\U(10)}_{(2^4, 1^2)}(\mathbf
1) = 29700$, where $\mathbf 1 =(1,1, \dots , 1)$. Actually, there
are $15$ Schur functions with negative coefficient $-1$ in the right
hand side of~\eqref{eq: sl2 ex10} including $S^{\U(10)}_{(2^4,
1^2)}(\mathbf x)$. After amazing cancellations among Schur functions, we obtain a symmetric function in the
left hand side of~\eqref{eq: sl2 ex10}, which contains only $1024$
monomial terms with coefficients all $1$ in its expansion.
\end{example}

\subsection{Subgroup $\langle \U(1), J, \pmb \zeta_4 \rangle$}
Set
$$
\pmb \zeta_4 := \left( \begin{matrix}
\sqrt{-1} & 0  \\
0 & \sqrt{-1}  \\
\end{matrix}\right) \in \U(2),
$$
and denote by $H_{2,4}$ the subgroup of $\U(2)$ generated by $\U(1)$, $J$ and $\pmb \zeta_4$.
That is, we define
\[  H_{2,4}:= \langle \U(1), J , \pmb \zeta_4 \rangle \le \U(2). \]
Then we have \[H_{2,4} = \U(1) \sqcup J\U(1) \sqcup \pmb \zeta_4 \U(1) \sqcup \pmb \zeta_4 J \U(1). \]
Let $H'_{2,4}$ be the subgroup of $H_{2,4}$ generated by $U(1)$ and $\pmb \zeta_4$.

\begin{theorem} \label{thm:g24}
For any partition $(a,b) \trianglelefteq (m^2)$, we have
\[  \m_{(a,b)}(H'_{2,4})= \delta ( a+b \equiv_4 0) \quad \text{ and } \quad \mathfrak m_{(a,b)}(H_{2,4})= \delta ( a+b \equiv_4 0)\delta ( a-b \equiv_4 0). \]
Moreover, for any $m \in \Z_{\ge1}$, we have
\begin{equation}\label{eq: g24 general}
\begin{aligned}
\dfrac{1}{2} \left( \prod_{i=1}^m (1+x_i^2)+ \prod_{i=1}^m (1-x_i^2) \right) & =
\sum_{\substack{(a,b) \trianglelefteq (m^2) \\ a+b \equiv_4 0} }(-1)^{ \delta(a \not\equiv_4 b)} S^{\U(m)}_{(2^{b}, 1^{a-b})}(\mathbf x) \\ &= \sum_{\substack{  (b+2j,b) \trianglelefteq (m^2)    \\   b+j \equiv_2 0 }} (-1)^{j} S^{\U(m)}_{(2^{b}, 1^{2j})}(\mathbf x) ,
\end{aligned}
\end{equation}
where we set $j := (a-b)/2$ as before.
\end{theorem}

\begin{proof}

Let $\lambda'=(a,b) \trianglelefteq (m^2)$. We keep the notations in the proof of Theorem \ref{thm :g21} for $V(\lambda') \cong \mathbf{det}^b \otimes \mathrm{Sym}^{a-b}(V)$. A vector in $V(\lambda')$ is fixed by $\U(1)$ if it is of the form $v_1^j v_2^j$ up to scalar multiplication. Since $\det (\pmb \zeta_4)=-1$, we get \[ \pmb \zeta_4 \, v_1^j v_2^j = (-1)^{b+j} v_1^j v_2^j .\] We see that $b+j \equiv_2 0 \Leftrightarrow a+b \equiv_4 0$, and if $a+b \equiv_4 0$ then $a-b \equiv_2 0$. Thus we obtain
\[   \m_{(a,b)}(H'_{2,4})= \delta ( a+b \equiv_4 0). \]
As observed in the proof of Theorem \ref{thm :g21}, the vector $v_1^j v_2^j$ is fixed by $J$ if and only if $a-b \equiv_4 0$. Therefore we have \[  \mathfrak m_{(a,b)}(H_{2,4})= \delta ( a+b \equiv_4 0)\delta ( a-b \equiv_4 0).\]

Let $d\gamma'= d\gamma'(\gamma)$ be the probability Haar measure on $H'_{2,4}$. Since $$ \det(I+xJ\gamma) = 1+x^2 \quad\text{ and } \quad    \det(I+x \, \pmb \zeta_4J\gamma) = 1-x^2$$ for all $\gamma  \in \U(1)$, we have
\begin{align*}
\int_{H_{2,4}} \Delta (\gamma) d\gamma &= \dfrac{1}{2} \int_{H'_{2,4}} \Delta (\gamma)   d\gamma' + \dfrac{1}{4}\int_{J\U(1)} \Delta (\gamma) du + \dfrac{1}{4}\int_{\pmb \zeta_4 J \U(1)} \Delta (\gamma) du\\
&= \dfrac{1}{2} \int_{H'_{2,4}} \Delta (\gamma) d\gamma'  + \dfrac{1}{4}\prod_{i=1}^m  (1+x_i^2)+\dfrac{1}{4} \prod_{i=1}^m  (1-x_i^2),
\end{align*}
where we write $\Delta(\gamma)=\prod_{i=1}^m \det(I+x_i \gamma)$ for convenience.
Applying Proposition \ref{prop-main} to the integrals, we obtain
\begin{align*} &\dfrac{1}{2} \left( \prod_{i=1}^m (1+x_i^2)+ \prod_{i=1}^m (1-x_i^2) \right)=
\sum_{(a,b) \trianglelefteq (m^2)}\left ( 2  \m_{(a,b)}(H_{2,4})-  \m_{(a,b)}(H'_{2,4}) \right )S^{\U(m)}_{(2^{b}, 1^{a-b})}(\mathbf x)\\&=
\sum_{\substack{(a,b) \trianglelefteq (m^2) \\ a+b \equiv_4 0} }(-1)^{ \delta(a \not\equiv_4 b)} S^{\U(m)}_{(2^{b}, 1^{a-b})}(\mathbf x)= \sum_{\substack{  (b+2j,b) \trianglelefteq (m^2)    \\   b+j \equiv_2 0 }} (-1)^{j} S^{\U(m)}_{(2^{b}, 1^{2j})}(\mathbf x) . \qedhere
\end{align*}
\end{proof}

\begin{remark} \label{rmk: g24} (1) The identity \eqref{eq: g24 general} can be derived from \eqref{eq: g2 general}. We will consider the alternate proof in Section \ref{alt}.

\noindent
(2) As in Remark~\eqref{rmk: g2}, we observe that the left hand side of~\eqref{eq: g24 general} is a simple combination of the monomial symmetric functions in $m$-variables:
$$ \dfrac{1}{2} \left( \prod_{i=1}^m (1+x_i^2)+ \prod_{i=1}^m (1-x_i^2) \right)=
\sum_{k=0}^{ \lfloor \frac{m}{2} \rfloor } \mathsf{m}^{(m)}_{(2^{2k})}(\mathbf x).$$
\end{remark}

\subsection{Pieri's rule} \label{alt}
One can see that
\begin{equation} \label{alP} \prod_{i=1}^m (1-x_i^2)   = \sum_{\substack{ \la \trianglelefteq (2^m) \\ \la'=(b+2j,b)}} (-1)^b \ S^{\U(m)}_{(2^b, 1^{2j})}(\mathbf x)   = \sum_{k=0}^m (-1)^k\mathsf{m}^{(m)}_{(2^k)}(\mathbf x).\end{equation}
Indeed, by replacing $x_i$ with $\sqrt{-1}\, x_i$ in \eqref{eq: g2 general}, we obtain
\[\prod_{i=1}^m (1-x_i^2)   = \sum_{\substack{ \la \trianglelefteq (2^m) \\ \la'=(b+2j,b)}} (-1)^j  (-1)^{j+b} \ S^{\U(m)}_{(2^b, 1^{2j})}(\mathbf x)= \sum_{\substack{ \la \trianglelefteq (2^m) \\ \la'=(b+2j,b)}} (-1)^b \ S^{\U(m)}_{(2^b, 1^{2j})}(\mathbf x) ,\]
and the second equality in \eqref{alP} follows from the definition of the monomial symmetric function $\mathsf m^{(m)}_{(2^k)}$. Combining \eqref{eq: g2 general} with \eqref{alP} yields the identity
\eqref{eq: g24 general}.
If we combine \eqref{eq: g2 general} and \eqref{alP} in a different way, we obtain
\begin{align*}
\frac{1}{2} \left(\prod_{i=1}^m (1+x_i^2) - \prod_{i=1}^m (1-x_i^2) \right) = \sum_{\substack{ (a , b)  \trianglelefteq ( m^2) \\  a+b \equiv_4 2 }} (-1)^{\delta(a \not\equiv_4 b)} S^{\U(m)}_{(2^{b}, 1^{a-b})}(\mathbf x) = \sum_{k=0}^{ \lfloor \frac{m-1}{2} \rfloor } \mathsf{m}^{(m)}_{(2^{2k+1})}(\mathbf x).
\end{align*}

Moreover, \eqref{alP} can also be proven using Pieri's rule. Since the idea can be used in other cases, let us see the proof.
We first consider the trivial case $g=1$ to use the results for the case $g=2$. Let $\mathsf m_\lambda^{(m)}$ be the monomial symmetric function of $m$-variables associated to a partition $\lambda$ with $\ell (\lambda) \le m$, as before. Then, in particular, we have
$$    \mathsf{m}_{(1^k)}^{(m)}= S_{(1^k)}^{\U(m)}\quad \text{ for } k \le m $$
and obtain
\begin{align} \label{eq: g=1 +}
\prod_{i=1}^m (1+x_i) = \sum_{k=0}^m   \mathsf{m}^{(m)}_{(1^k)} (\mathbf x) = \sum_{\la \trianglelefteq (1^m)} S_{\la}^{\U(m)}(\mathbf x).
\end{align}
By replacing $x_i$ with $-x_i$, we have
\begin{align} \label{eq: g=1 -}
\prod_{i=1}^m (1-x_i) = \sum_{k=0}^m  (-1)^k \mathsf{m}^{(m)}_{(1^k)} (\mathbf x) = \sum_{\la \trianglelefteq (1^m)} (-1)^{|\la|} S_{\la}^{\U(m)}(\mathbf x).
\end{align}

Recall Pieri's rule from, e.g., Macdonald's book \cite[(5.17)]{Mac98}:
\begin{align}    \label{eq: Pieri}
 S^{\U(m)}_{\la}(\mathbf x) \times \prod_{i=1}^m (1+x_i) &=   S^{\U(m)}_{\la}(\mathbf x) \times \left(   \sum_{l=0}^m \mathsf{e}_l (\mathbf x)  \right) =
\sum_{ \la \trianglelefteq \mu \trianglelefteq \la+(1^m)} S^{\U(m)}_{\mu}(\mathbf x),
\end{align}
where $\mathsf{e}_l(\mathbf x) $ denotes the elementary symmetric function of partition $(l)$ of length $1$ and $\lambda + (1^m) = (\lambda_1+1, \lambda_2 +2 , \dots , \lambda_m+1)$ for $\lambda = (\lambda_1, \lambda_2, \dots , \lambda_m)$.

By~\eqref{eq: Pieri} and~\eqref{eq: g=1 -}, we have
\begin{equation}\label{eq: process 1-x_i2}
\begin{aligned}
\prod_{i=1}^m (1-x_i^2)  & = \prod_{i=1}^m (1+x_i) \times \sum_{ (1^k) \trianglelefteq (1^m)} (-1)^{k} S_{(1^k)}^{\U(m)}(\mathbf x) \\
& =\sum_{k=0}^m  \sum_{ (1^k) \trianglelefteq \mu \trianglelefteq (2^k,1^{m-k})  } (-1^{k})  S^{\U(m)}_{\mu}(\mathbf x)
\end{aligned}
\end{equation}
Thus for any partition $(2^b,1^{a-b}) \trianglelefteq (2^m)$, the coefficient of $S^{\U(m-1)}_{(2^b,1^{a-b})}(\mathbf x)$ in~\eqref{eq: process 1-x_i2} is given by
$$
\sum_{s=b}^{a} (-1)^s  = \begin{cases}
0 &\text{ if } a-b \equiv_2 1, \\
1 &\text{ if } a-b \equiv_2 0 \text{ and } b \equiv_2 0, \\
-1 &\text{ if } a-b \equiv_2 0 \text{ and } b \equiv_2 1. \\
\end{cases}
$$
When $a-b \equiv_2 0$, write $a-b=2j$. Then we have
\begin{align*}
\prod_{i=1}^m (1-x_i^2)  & = \sum_{\substack{ \la \trianglelefteq (2^m) \\ \la'=(b+2j,b)}} (-1)^b \ S^{\U(m)}_{(2^b, 1^{2j})}(\mathbf x)
\end{align*} as desired.

Using a similar argument, we obtain two more identities.

\begin{prop} \label{prop: byproduct}
For any $m \in \Z_{\ge1}$, we have
\begin{align}\label{eq: (1+x)2}
\prod_{i=1}^m (1+2x_i+x_i^2) &=  \sum_{ (a , b)  \trianglelefteq ( m^2)} (a-b+1)  S^{\U(m)}_{(2^{b}, 1^{a-b})}(\mathbf x) , \\
\label{eq: (1-x)2}
\prod_{i=1}^m (1-2x_i+x_i^2) &=  \sum_{  (a , b)  \trianglelefteq ( m^2)}      (-1)^{a-b} (a-b+1)  S^{\U(m)}_{(2^{b}, 1^{a-b})}(\mathbf x).
\end{align}
\end{prop}

\begin{proof}
By~\eqref{eq: Pieri} and~\eqref{eq: g=1 +}, we have
\begin{equation}\label{eq: process (1+x_i)2}
\begin{aligned}
\prod_{i=1}^m (1+x_i)^2  & = \prod_{i=1}^m (1+x_i) \times \sum_{ (1^k) \trianglelefteq (1^m)}  S_{(1^k)}^{\U(m)}(\mathbf x) \\
& =\sum_{k=0}^m \sum_{ (1^k) \trianglelefteq \mu \trianglelefteq (2^k,1^{m-k})  }   S^{\U(m)}_{\mu}(\mathbf x)
\end{aligned}
\end{equation}
Thus for any partition $(2^b,1^{a-b}) \trianglelefteq (2^m)$, the coefficient of $S^{\U(m)}_{(2^b,1^{a-b})}(\mathbf x)$ in~\eqref{eq: process (1+x_i)2} is equal to
$$
\sum_{s=b}^{a} 1 =  a-b+1,
$$
and the identity \eqref{eq: (1+x)2} follows.
By replacing $x_i$ with $-x_i$ in~\eqref{eq: (1+x)2}, we obtain \eqref{eq: (1-x)2}.
\end{proof}

\begin{remark} \label{rmk: may not applicable}
The above use of Pieri's rule may not be applicable, in general, if one can try to obtain \emph{closed-form formulas} for $g \ge 3$.
For instance, based on Theorem~\ref{thm :g21} about $g=2$, one can check the formula in Theorem~\ref{thm :g3} about $g=3$ below using Pieri's rule, for \emph{first several small values of $m$}.
But, when $g\ge 3$, obtaining closed-form formula for the coefficient of $S_{\la}^{\U(m)}$, $m \ge 1$,  seems not easy in this approach. 
\end{remark}

\section{Identities for $g=3$} \label{sec: g=3}
In this section, we consider some disconnected subgroups $H$ of $\U(3)$ and compute $\mathfrak m_{\lap}(H)$ for $\la \trianglelefteq (3^m)$ in \eqref{main-f}. As with the case $g=2$, our computation yields identities involving Schur functions $S_{\lambda}^{\U(m)}$ for all $m \in \mathbb Z_{\ge 1}$.

\medskip

To begin with, we embed $\U(1)$ into $\U(3)$ via
$$U(1) \simeq \left \{  \begin{pmatrix} t & 0 & 0 \\ 0 & t^{-1} & 0  \\ 0 &0 &1 \end{pmatrix} : t \in \C, |t|=1 \right \},$$
and $\U(2)$ into $\U(3)$ via $ A \mapsto \begin{pmatrix} A& 0 \\ 0 & 1 \end{pmatrix}$ for $A \in \U(2)$.

\subsection{Subgroup $\langle \U(1), J \rangle$}
Define
$$
J: = \left(
\begin{matrix}
0 & 1 & 0 \\ -1 & 0 & 0 \\  0& 0 & 1
\end{matrix}
\right) \in \U(3).
$$
Note that
$$J^2 =\left(\begin{matrix}
-1 & 0 & 0 \\ 0 &-1 & 0 \\  0 & 0 & 1
\end{matrix} \right) \in \U(1) \quad \text{ and } \quad J^4 = \left(\begin{matrix}
1 & 0 & 0 \\ 0 & 1 & 0 \\  0 & 0 & 1
\end{matrix} \right) =I.$$

Consider the subgroup $H_3$ of $\U(3)$ generated by $J$ and $\U(1)$, i.e.
\[ H_3 : = \langle \U(1), J \rangle \le \U(3) .\] Then one can easily check that
$J$ normalizes $\U(1)$, and $H_3 = \U(1) \sqcup J\U(1)$.
Note that
\begin{equation} \label{eqn-J} J   \left( \begin{matrix} t & 0 & 0 \\ 0 &   t^{-1}& 0 \\ 0 & 0 & 1\end{matrix}\right) = \left( \begin{matrix} 0 &  t^{-1}&0 \\ -t &0 & 0 \\
 0 & 0 & 1  \end{matrix}\right) \quad \text{ and } \quad
\det(I+xJ\gamma) = (1+x)(1+x^2)
\end{equation}
for all $\gamma \in \U(1)$.

\medskip

We prove a useful lemma.
\begin{lem} \label{lem-sh}
For any $k \in \mathbb Z_{\ge 0}$, we have
\[  \m_{(a+k, b+k,k)}(\U(1)) = \m_{(a,b,0)}(\U(1)) \quad \text{ and } \quad \m_{(a+k, b+k,k)}(H_3) = \m_{(a,b,0)}(H_3). \]
\end{lem}

\begin{proof}
Let $\mathbf{det}$ be the one-dimensional representation of $\U(3)$ defined by the determinant. Then we have \[ V(a+k,b+k,k) = \mathbf{det}^k \otimes V(a,b,0) .\] Since $\det(A)=1$ for any $A \in \U(1)$ and $\det(J)=1$, the assertion follows.
\end{proof}

Thanks to Lemma \ref{lem-sh}, we need to consider the irreducible representations $V(\lambda')$ of $\U(3)$ only for $\lap=(a,b,0)$. In what follows, we assume $\lap=(a,b,0)$ and freely write $V(a,b)=V(\lap)$.
Note that we may also regard $V(a,b)$ as the irreducible representation of $\mathfrak {sl}(3,\mathbb C)$ with highest weight $(a,b)$.
More precisely, define $h_1, h_2 \in \mathfrak{sl}(3, \mathbb C)$ by
\[ h_1 = \mathrm{diag}(1,-1,0) \quad \text{ and } \quad h_2 = \mathrm{diag}(0,1,-1) ,\]
and denote by $\mathfrak h$ the subspace of $\mathfrak {sl}(3, \mathbb C)$ spanned by $h_1$ and $h_2$. We regard any partition $\mu=(\mu_1, \mu_2)$ as an element of $\mathfrak h^*$ by setting
\[  \mu(h_1) = \mu_1 - \mu_2 \quad \text{ and } \quad \mu(h_2) = \mu_2 ,\] and $\mu$ is a weight of $\mathfrak {sl}(3, \mathbb C)$.

Let $V$ and $W= \bigwedge^2 V$ be the fundamental representations of $\U(3)$. Take a basis $\{ v_1, v_2,v_3\}$  of $V$ such that
\begin{equation}
\label{eqn-Jv} Jv_1=-v_2, \quad Jv_2= v_1 \quad \text{ and }  \quad Jv_3=v_3.
\end{equation}
Write
\[ w_{12} = v_1 \wedge v_2, \quad  w_{13}=v_1 \wedge v_3, \quad w_{23}=v_2 \wedge v_3. \]
Then  $\{w_{12}, w_{13}, w_{23} \}$ is a basis for $W$, and we have
\begin{align*}
J w_{12} = w_{12}, \qquad J w_{13} = -w_{23} \quad \text{ and } J w_{23} = w_{13}.
\end{align*}

As $\mathfrak{sl}(3, \mathbb C)$-representations, $V$ and $W$ are equivalent and can be described as follows:
\begin{equation} \label{eqn-dia}{\normalsize
V :\, \tyg(1) \overset{1}{\to}  \tyg(2) \overset{2}{\to}  \tyg(3), \quad \quad W :\, \tyg(1,2) \overset{2}{\to}  \tyg(1,3) \overset{1}{\to}  \tyg(2,3).}
\end{equation}
The diagrams mean
\[ f_1 v_1 = v_2, \quad f_2 v_2=v_3, \quad f_2 w_{12}=w_{13} , \quad f_1 w_{13}=w_{23}\] and $f_i v_j=0$ and $f_i w_{jk}=0$ for other choices of $i,j,k$,
where we set $f_1: =\scriptsize{\begin{pmatrix} 0&0&0\\1&0&0\\0&0&0 \end{pmatrix}}, \ f_2 :=\scriptsize{\begin{pmatrix} 0&0&0\\0&0&0\\0&1&0 \end{pmatrix}} \in \mathfrak{sl}(3,\mathbb C)$.
Equivalently, if we set $e_1: =\scriptsize{\begin{pmatrix} 0&1&0\\0&0&0\\0&0&0 \end{pmatrix}},\ e_2 :=\scriptsize{\begin{pmatrix} 0&0&0\\0&0&1\\0&0&0 \end{pmatrix}} \in \mathfrak{sl}(3,\mathbb C)$, we have
\[ e_2 v_3 = v_2, \quad e_1 v_2=v_1, \quad e_1 w_{23}=w_{13} , \quad e_2 w_{13}=w_{12} \]
and $e_i v_j=0$ and $e_i w_{jk}=0$ for other choices of $i,j,k$. The embedding $\U(2) \hookrightarrow \U(3)$ corresponds to $\langle e_1,h_1,f_1 \rangle \cong \mathfrak{sl}(2,\mathbb C) \hookrightarrow \mathfrak{sl}(3,\mathbb C)$.

We realize the representation $V(a,b)$ for a partition $(a,b)$ as the irreducible component of $\mathrm{Sym}^{a-b}\, V \otimes \mathrm{Sym}^b\, W$ generated by the highest weight vector $v_1^{a-b} \otimes w_{12}^b$. In particular, $V= V(1,0)$ and $W= V(1,1)$.
We identify $V(a+1,b+1)$ with the image of the embedding
\begin{align}\label{eq: iota}
 \iota_{a+1,b+1} :V(a+1,b+1) \hookrightarrow V(a,b) \otimes V(1,1)
\end{align}
given by
\[v_1^{a-b} \otimes w_{12}^{b+1} \mapsto (v_1^{a-b} \otimes w_{12}^b )\otimes w_{12}. \]

For a partition $(a,b)$, set $z:= a-b$ and define a set of partitions which interlace with $(a,b)$:
$$   \Phi(a,b) \seteq \{ (p,q) \ | \ a \ge p \ge b \ge q \ge 0\}.$$
Clearly, we have
$$  | \Phi(a,b)|  = (a-b+1)(b+1)= (z+1)(b+1).$$
It follows from Proposition \ref{thm: branching A2 to A1} that
\begin{equation} \label{Phiab}  \text{the set $\Phi(a,b)$ is exactly the set of $\U(2)$-highest weights in the restriction $V(a,b) |_{\U(2)}$.} \end{equation}

\begin{example} \hfill
\begin{enumerate}
\item[{\rm (a)}] $\Phi(3,2) = \{ (3,0),(2,0),(3,1),(2,1),(3,2),(2,2)  \}$.
\item[{\rm (b)}]$\Phi(4,3) = \{ (4,0),(3,0),(4,1),(3,1),(4,2),(3,2),(4,3),(3,3)  \}$.
\end{enumerate}
\end{example}

For our purpose, we need to precisely describe $\U(2)$-highest weight vectors in the restriction $V(a,b)|_{\U(2)}$. In what follows, we specify such vectors. We freely use the $\mathfrak{sl}(3, \mathbb C)$-representation structure on $V(a,b)$ and apply actions of $e_i, f_i$ $(i=1,2)$ on vectors of $V(a,b)$.

For $V(a,0)$, define
$$v_{(k,0;a,0)}\seteq v_1^kv_3^{a-k} \quad  \text{ for } (k,0) \in \Phi(a,0) \quad (\text{or equivalently, for } 0 \le k \le a ).$$
By considering $\mathfrak{sl}(2, \mathbb C)$ action from \eqref{eqn-dia}, one can see that $v_{(k,0;a,0)}$ are $\U(2)$-highest weight vectors with highest weights $(k,0)$.

Next let us consider $\U(2)$-highest weight vectors
of $V(a+1,1)$ via the embedding  $\iota_{a+1,1}$ in~\eqref{eq: iota}. Obviously, the vectors $$v_{(k+1,1; a+1,1)}:=v_{(k,0;a,0)} \otimes w_{12}$$ are contained in
$V(a+1,1)$ from the construction of $V(a+1,1)$, and each of them is a $\U(2)$-highest weight vector of $V(a+1,1)$, which generates a $(k+1)$-dimensional $\U(2)$-module.
Thus they correspond to the partitions $(k+1,1)$ in $\Phi(a+1,1)$.
Similarly, the vectors
$$v_{(k+1,0; a+1,1)}:=f_2 \left ( v_{(k,0;a,0)} \otimes w_{12} \right )= v_{(k,0;a,0)} \otimes w_{13}$$
are contained in $V(a+1,1)$ and are $\U(2)$-highest weight vectors of $V(a+1,1)$. Each of them generates a $(k+2)$-dimensional $\U(2)$-module. Hence they correspond to the partitions $(k+1,0)$ in $\Phi(a+1,1)$. Since
$$\Phi(a+1,1) =   \{ (k+1,1) \ | \  0 \le k \le a  \}\sqcup \{ (k+1,0) \ | \  0 \le k \le a  \},$$ we have obtained all the $\U(2)$-highest weight vectors of the restriction $V(a+1,1)|_{\U(2)}$.

Generally, for $(p,q) \in \Phi(a,b)$, define
\[ v_{(p,q;a,b)}:=v_1^{p-b} v_3^{a-p} \otimes w_{12}^q w_{13}^{b-q}, \qquad b\le p \le a .\]

\begin{lem} \label{lemU2h}
The vectors $ v_{(p,q;a,b)}$ are $\U(2)$-highest weight vectors of $V(a,b)$ for $(p,q) \in \Phi(a,b)$.
\end{lem}

\begin{proof}
It is straightforward to check that $v_{(p,q;a,b)}$ are $\U(2)$-highest weight vectors. For induction, assume that $v_{(p,q;a-1,b-1)}$ are contained in $V(a-1,b-1)$. Then, by the construction of $V(a,b)$, we have $$v_{(k+b-1,b-1; a-1,b-1)} \otimes w_{12} =\left ( v_1^k v_3^{a-b-k} \otimes w_{12}^{b-1} \right ) \otimes w_{12} = v_1^k v_3^{a-b-k} \otimes w_{12}^{b}=v_{(k+b,b;a,b)} \in V(a,b)$$  for $0 \le k \le a-b$.
Since \[f_2^lv_{(k+b,b;a,b)} = \tfrac {b!}{(b-l)!} v_1^k v_3^{a-b-k} \otimes w_{12}^{b-l}w_{13}^l =  \tfrac {b!}{(b-l)!} v_{(k+b,b-l;a,b)} \quad \text{ for } 0 \le l \le b,  \]
we have $v_{(p,q;a,b)} \in V(a,b)$ for any $(p,q) \in \Phi(a,b)$. \end{proof}

The vectors in the above lemma are distinct, linearly independent and exhaust all the $\U(2)$-highest weight vectors in $V(a,b)$.
Note that
\begin{eqnarray}&&
\parbox{95ex}{
\begin{enumerate}
\item[(J1)] \label{it: rel1}
Since $J(v_{(p,q;a,b)})=J(v_1^{p-b}v_3^{a-p} \otimes w_{12}^{q} w_{13}^{b- q})  = (-1)^{p-q}     v_2^{p-b}v_3^{a-p} \otimes w_{12}^{q} w_{23}^{b- q}
$, the vector $J(v_{(p,q;a,b)})$
is a $\U(2)$-lowest vector in the $\U(2)$-representation generated by $v_{(p,q;a,b)}$.
\item[(J2)]\label{it: rel2} $J f_1 = -e_1 J$.
\end{enumerate}
}\label{eq: J condition}
\end{eqnarray}
\begin{prop} \label{proplap}
For a partition $\lap=(a,b,0)$, the multiplicity $\m_\lap(\U(1))$ of the trivial representation in $V(\lap)|_{\U(1)}$ is equal to the cardinality of the set
 \[ \Phi^{(2)}(a,b):= \{ (p,q) \in \Phi (a,b)   \ | \  p \equiv_2 q \},\] and the multiplicity $\m_\lap(H_3)$ of the trivial representation in $V(\lap)|_{H_3}$ is equal to the cardinality of the set
$$\Phi^{(4)}(a,b):= \{ (p,q) \in   \Phi(a,b)  \ | \  p \equiv_4 q \}.$$
\end{prop}

\begin{proof}
From the embeddings of $\U(1)$ into $\U(3)$, we see that the multiplicity $\m_\lap(\U(1))$ is equal to the number of linearly independent vectors in $V(\lap)$ with weight $\mu$ such that
$$  \mu(h_1) =0 .$$ Similarly, the multiplicity $\m_\lap(H_3)$ is equal to the number of linearly independent vectors $v$ in $V(\lap)$ with weight $\mu$ such that
$$  \mu(h_1) =0 \quad  \text{ and }  \quad Jv=v.$$

If we consider the restriction $V(\lap)|_{\U(2)}$, then the condition $\mu(h_1) =0$ means weight $0$.
A weight $0$ vector occurs in $V_2(p,q)$ precisely when $p-q \equiv_2 0$ with multiplicity 1, where $V_2(p,q)$ is the irreducible representation of $\U(2)$ with highest weight $(p,q)$. Thus the first assertion follows from \eqref{Phiab}.

Write $p-q = 2k$. Then $v\seteq f_1^k v_{(p,q;a,b)}$ is a weight $0$ vector by Lemma \ref{lemU2h}. Using (J1) and (J2) in~\eqref{eq: J condition}, we obtain
\begin{align*}
J(v)&=J f_1^k (v_1^{p-b}v_3^{a-p} \otimes w_{12}^{q} w_{13}^{b- q})
 = (-1)^k  e_1^k J(v_1^{p-b}v_3^{a-p} \otimes w_{12}^{q} w_{13}^{b- q})   \\
& = (-1)^{3k} e_1^k (v_2^{p-b}v_3^{a-p} \otimes w_{12}^{q} w_{23}^{b- q}) = (-1)^{3k}  f_1^k (v_1^{p-b}v_3^{a-p} \otimes w_{12}^{q}w_{13}^{b- q})   = (-1)^{k} v .
\end{align*}
Thus $v$ is fixed only when $k$ is even. Thus the second assertion follows.
\end{proof}

The cardinalities of the sets $\Phi^{(2)}$ and  $\Phi^{(4)}$ are computed in the following proposition.
\begin{prop} \label{prop-card}
For a partition $(a,b)$, write $z:= a-b$.
 Then we have
\[ | \Phi^{(2)} (a,b) |  = \left \lceil (z+1)(b+1)/ 2 \right \rceil, \]
and
$$
 | \Phi^{(4)} (a,b) |  =
 \dfrac{1}{2}\left(\left\lceil (z+1)(b+1)/{2}\right\rceil  + \tau(z,b) \right)
$$
where
$\tau(z,b) \in \{0,\pm1\}$ is defined on the congruence classes of $z$ and $b$ modulo $4$ as follows:
\begin{center} $\tau(z,b)=$ {\scriptsize
\begin{tabular}{|c||c|c|c|c|}
\hline
$z\backslash b$ & 0&1&2&3\\ \hline \hline
0&$1$ & $1$&$0$& $0$ \\ \hline
1&$1$&$0$ & $-1$ & $0$ \\ \hline
2&$0$&$-1$ & $-1$ & $0$ \\ \hline
3&$0$&$0$ & $0$ & $0$ \\ \hline
\end{tabular}}
 \end{center}
\end{prop}

\begin{proof}
The elements $(p,q)$ in $\Phi(a,b)$ and the corresponding dimensions $p-q+1$ can be each arranged into an array of size
$(z+1) \times (b+1)$ as follows, where we put $(p,q)$ in the left and its   corresponding dimensions  in the right:
\begin{equation} \label{eq: array and dimension}
{\scriptsize
\left.
\begin{matrix}
(a,0) & (a,1) & \cdots & (a,b) \\
(a-1,0) & (a-1,1) & \cdots & (a-1,b) \\
\vdots & \vdots & \cdots & \vdots  \\
(b,0) & (b,1)& \cdots & (b,b)
\end{matrix} \ \ \right|
\quad
\begin{matrix}
a+1 & a & \cdots & a-b+1 \\
a & a-1 & \cdots & a-b \\
\vdots & \vdots & \cdots & \vdots  \\
b+1 & b  & \cdots & 1
\end{matrix} }
\end{equation}
By counting the number of odd integers in the right array, we obtain
\[ | \Phi^{(2)} (a,b) |  = \left \lceil (z+1)(b+1)/ 2 \right \rceil, \]
and  by counting the number of integers congruent to 1 modulo $4$ in the right array,
we get
\[
 | \Phi^{(4)} (a,b) |  =
 \dfrac{1}{2}\left(\left\lceil(z+1)(b+1)/{2}\right\rceil  + \tau(z,b) \right).  \qedhere
\]
\end{proof}

Now we state and prove the main theorem of this subsection.

\begin{theorem} \label{thm :g3}
For a partition $\lap=(a+k,b+k, k)$, $k \in \mathbb Z_{\ge 0}$, we have
\[  \m_{\lap}(\U(1))= \left \lceil (z+1)(b+1)/ 2 \right \rceil \quad \text{ and } \quad \mathfrak m_{\lap}(H_3)=  \dfrac{1}{2}\left(\left\lceil(z+1)(b+1)/{2}\right\rceil  + \tau(z,b) \right) , \] where we set $z:=a-b$.
Furthermore, for any $m \in \Z_{\ge1}$, we have
\begin{align}\label{eq: g3 general}
\prod_{i=1}^m (1+x_i)(1+x_i^2) & =\sum_{\substack{\la \trianglelefteq (3^m) \\ \la=(3^k,2^b,1^z)}} \tau(z,b)\,  S^{\U(m)}_{\la}(\mathbf x), \\
 \prod_{i=1}^m (1-x_i)(1+x_i^2)&= \sum_{\substack{\la \trianglelefteq (3^m) \\ \la=(3^k,2^b,1^z)}} (-1)^{|\la|}  \tau(z,b)  S^{\U(m)}_{\la}(\mathbf x). \label{eq:- g3 general}
\end{align}

\end{theorem}

\begin{proof}
From Proposition \ref{proplap}, we obtain
\[   \m_{\lap}(\U(1))= | \Phi^{(2)} (a,b) | \quad \text{ and } \quad   \m_{\lap}(H_3)=| \Phi^{(4)} (a,b) |, \]
and the formulas for $\m_{\lap}(\U(1))$ and $\m_{\lap}(H_3)$ are from Proposition \ref{prop-card}.

Let $d\gamma, d \gamma_1$ be the probability Haar measures on $H_3, \U(1) \le \U(3)$, respectively.  Since $H_{3}=\U(1) \sqcup J\U(1)$, we use \eqref{eqn-J} and Proposition \ref{prop-main} to obtain
\begin{align*}
 \sum_{\la \trianglelefteq (3^m)}\m_{\lap}(H_3)  S^{\U(m)}_{\la}(\mathbf x)&=\int_{H_3} \prod_{i=1}^m  \det(I+x_i\gamma) d\gamma \\ & = \dfrac{1}{2} \int_{\U(1)}  \prod_{i=1}^m  \det(1+x_i\gamma) d\gamma_1
+  \dfrac{1}{2}  \int_{J\U(1)} \prod_{i=1}^m (1+x_i)(1+x_i^2) d\gamma_1 \allowdisplaybreaks\\
& = \dfrac{1}{2}  \sum_{\la \trianglelefteq (3^m)}\m_{\lap}(\U(1))  S^{\U(m)}_{\la}(\mathbf x) +  \dfrac{1}{2}   \prod_{i=1}^m (1+x_i)(1+x_i^2).
\end{align*}
Hence, using Lemma \ref{lem-sh},
\begin{align*}
 \prod_{i=1}^m (1+x_i)(1+x_i^2)=   \sum_{\la \trianglelefteq (3^m)}\left ( 2 \m_{\lap}(H_3) - \m_{\lap}(\U(1)) \right )  S^{\U(m)}_{\la}(\mathbf x)  =
\sum_{\substack{\la \trianglelefteq (3^m) \\ \la=(3^k,2^b,1^z)}} \tau(z,b)  S^{\U(m)}_{\la}(\mathbf x).
\end{align*}

The identity \eqref{eq:- g3 general} follows from \eqref{eq: g3 general} by replacing $x_i$ with $-x_i$.
\end{proof}

\begin{remark} \label{rmk: g3} The left hand side of~\eqref{eq: g3 general} is a simple combination of the monomial symmetric functions $\mathsf m^{(m)}_\la$ associated with partitions $\la$ in $m$-variables:
$$ \prod_{i=1}^m (1+x_i)(1+x_i^2)= \sum_{ \la  \trianglelefteq (3^m) } \mathsf{m}^{(m)}_{\la}(\mathbf x).$$
\end{remark}

\begin{example} \label{ex: g3 huge}
(1) Let us see an example for the case $m = 7$ in Theorem~\ref{thm :g3}. We have
 \begin{align}\label{eq: sl3 ex7}
 \prod_{i=1}^7 (1+x_i)(1+x_i^2)=  \sum_{\substack{\la \trianglelefteq (3^7) \\ \la=(3^k,2^b,1^z)}} \tau(z,b)  S^{\U(7)}_{\la}(\mathbf x)
\end{align}
One can see that $S^{\U(7)}_{(3^2,2^2, 1^{1})}(\mathbf x)$ appears in the right hand side of~\eqref{eq: sl3 ex7} with the coefficient $-1$ as
$\tau(1,2) = -1$. As a polynomial, $S^{\U(7)}_{(3^2,2^2, 1^{1})}(\mathbf x)$ contains $1778$ monomial terms and $S^{\U(7)}_{(3^2,2^2, 1^{1})}(\mathbf 1)= 7560$, where $\mathbf 1 =(1,1, \dots , 1)$.
There are $18$ Schur functions with coefficient $-1$ in the right hand side of~\eqref{eq: sl3 ex7} including $S^{\U(7)}_{(3^2,2^2,1^{1})}(\mathbf x)$. On the other hand, we obtain a polynomial in the left hand side of~\eqref{eq: sl3 ex7}, which contains $16384$
monomial terms with coefficient all $1$. One can check this, for example, using \textsc{SageMath}.

(2) Let us consider the case $m = 20$ in Theorem~\ref{thm :g3}. We have
 \begin{align}\label{eq: sl3 ex20}
 \prod_{i=1}^{20} (1+x_i)(1+x_i^2)=  \sum_{\substack{\la \trianglelefteq (3^{20}) \\ \la=(3^k,2^b,1^z)}} \tau(z,b)  S^{\U(20)}_{\la}(\mathbf x)
\end{align}
Here $S^{\U(20)}_{(2^6, 1^{5})}(\mathbf x)$ appears in the right hand side of~\eqref{eq: sl3 ex20} with the coefficient $-1$ as
$\tau(5,6)=\tau(1,2) = -1$. Using \textsc{WeylCharacterRing} in \textsc{SageMath},  one can check that $S^{\U(20)}_{(2^6, 1^{5})}(\mathbf 1)= 4557090720$.
Including $S^{\U(20)}_{(2^6, 1^{5})}$, there are $315$ Schur functions with coefficient $-1$ in the right hand side of~\eqref{eq: sl3 ex20} and
the specialization of the left hand side of~\eqref{eq: sl3 ex20} at $\mathbf 1$ is equal to $4^{20}=1099511627776$. However, checking
 whether the right hand side of~\eqref{eq: sl3 ex20} coincides with the left hand side of~\eqref{eq: sl3 ex20} may well go beyond the capacity of a regular personal computer.\end{example}

\subsection{Subgroup $\langle \U(1), J,\pmb \zeta_4 \rangle$}

Let us set
$$
\pmb \zeta_4 := \left( \begin{matrix}
 \sqrt{-1} & 0 & 0 \\
0 & \sqrt{-1} & 0 \\
0 & 0 & 1
\end{matrix}\right) \in U(3) \setminus U(1).
$$
Then we have
\begin{equation} \label{eqn-Jz}  \det(I+x\pmb \zeta_4J\gamma)
= (1+x)(1-x^2) \quad\text{ and } \quad  \det(I+xJ\gamma) = (1+x)(1+x^2).\end{equation}

Define $H_{3,4}$ to be the subgroup generated by $\U(1),J,\pmb \zeta_4$, i.e.
\[ H_{3,4} : = \langle \U(1), J , \pmb \zeta_4 \rangle  \le \U(3) . \]
Then we have the coset decomposition
\begin{equation} \label{H34} H_{3,4}= \U(1) \sqcup J\U(1) \sqcup \pmb \zeta_4 \U(1) \sqcup \pmb \zeta_4 J \U(1) .\end{equation}
Let $H'_{3,4}$ be the subgroup  of $H_{3,4}$ generated by $\U(1), \pmb \zeta_4$. Note that
\begin{equation} \label{Hp34} H_{3,4} = H'_{3,4} \sqcup J H'_{3,4} . \end{equation}

\begin{lem} \label{lem-sh-z}
For any $c \in \mathbb Z_{\ge 2}$, we have
\[  \m_{(a, b,c)}(H_{3,4}) = \m_{(a-2, b-2, c-2)}(H_{3,4}) \quad \text{ and } \quad \m_{(a, b,c)}(H'_{3,4}) = \m_{(a-2, b-2, c-2)}(H'_{3,4}) .\]
\end{lem}

\begin{proof}
As before, let $\mathbf{det}$ be the one-dimensional representation of $\U(3)$ defined by the determinant. Then we have \[ V(a,b,c) = \mathbf{det}^2 \otimes V(a-2,b-2,c-2) .\]
Since $\det(A)=1$ for any $A \in \U(1)$,  $\det(J)=1$ and $\det(\pmb \zeta_4)=-1$, the assertion follows.
\end{proof}

By Lemma \ref{lem-sh-z}, it  suffices to consider partitions of the form
$$ (a,b, \epsilon), \qquad \epsilon \in \{0,1 \}. $$
For a partition $(a,b,\epsilon)$ $(\epsilon \in \{0,1\})$, define
$$\Phi(a,b,\epsilon):= \{   (p,q) \in \Z^2 \ | \   a \ge p \ge b \ge q \ge \epsilon  \}.$$
For example,  $\Phi(3,2,1) = \{ (3,2),(2,2),(3,1),(2,1)  \}$. From Proposition~\ref{thm: branching A2 to A1}, we see that
\begin{equation} \label{Phiabe}  \text{the set $\Phi(a,b,\epsilon)$ is exactly the set of $\U(2)$-highest weights in the restriction $V(a,b,\epsilon) |_{\U(2)}$.} \end{equation}

\begin{prop} \label{prop-mu}
For a partition $\lap=(a,b,\epsilon)$ $(\epsilon \in \{0,1\})$, the multiplicity
$ \m_{\lap}(H'_{3,4})  $ is the same as the cardinality of the set
\begin{align*}
\Phi^{(2,4)}(\lap)  & = \{ (p,q) \in \Phi(\lap) \ | \       p+q  \equiv_4 0    \},
\end{align*}
and the multiplicity $ \m_{\lap}(H_{3,4})  $ is the same as the cardinality of the set
$$\Phi^{(4,4)}(\lap)  = \{ (p,q) \in \Phi(\lap) \ | \ p-q  \equiv_4 0, \; p+q \equiv_4 0    \}.$$
\end{prop}

\begin{proof}
Let $v$ be a weight vector fixed by $H'_{3,4}$. Then, in particular, $v$ is fixed by $\U(1)$, and we may write $v= f_1^k v_{(p,q;a,b)}$ for $p-q = 2k$ as in the proof of Proposition \ref{proplap}.
Since
\[ v=f_1^k ( v_1^{p-b}v_3^{a-p} \otimes w_{12}^q w_{13}^{b-q})= C v_1^{p-b-k} v_2^k v_3^{a- p} \otimes w_{12}^q w_{13}^{b-q}+ ( \text{other terms with the same weight} ),
\] for some constant $C$, the number of $v_1$ factors and the number of $v_2$ factors are the same, which is equal to $k+q= (p+q)/2$, and we have
\[ \pmb \zeta_4 v= (\sqrt{-1})^{(p+q)} v=v .\]
Thus we have $p+q \equiv_4 0$, which implies $p-q\equiv_2 0
$.

If $v$ is also fixed by $J$, the we obtain an additional condition $p-q \equiv_4 0$ as in (the proof of) Proposition \ref{proplap}.
\end{proof}

The cardinalities of the sets $\Phi^{(2,4)}(\lap)$ and $\Phi^{(4,4)}(\lap)$ will be computed in the next section, and we present the resulting formulas in the following two propositions.

\begin{prop} \label{prop-ca-1}
For a partition $\lap=(a,b,\epsilon)$ $(\epsilon \in \{0,1\})$,
write $z \seteq a-b$ and $b' \seteq b-\epsilon$. Then we have
\begin{align}\label{eq: 24 closed}
| \Phi^{(2,4)}(\lap) |  = \dfrac{(z+1)(b'+1)+\kappa_\epsilon(z,b')}{4},
\end{align}
where
$\kappa_\epsilon(z, b' )$ are defined on the congruence classes of $z$ and $b'$ modulo $4$  by
\begin{center}{ $\kappa_0(z,b) =$ {\scriptsize
\begin{tabular}{|c||c|c|c|c|}
\hline
$z\backslash b$ & 0&1&2&3\\ \hline \hline
0&$3$ & $-2$&$1$& $0$ \\ \hline
1&$2$&$-4$ & $2$ & $0$ \\ \hline
2&$1$&$-2$ & $3$ & $0$ \\ \hline
3&$0$&$0$ & $0$ & $0$ \\ \hline
\end{tabular}} \qquad  and  \qquad $\kappa_1(z,b') =$
 {\scriptsize \begin{tabular}{|c||c|c|c|c|}
\hline
$z\backslash b'$ & 0&1&2&3\\ \hline \hline
0&$-1$ & $2$&$1$& $0$ \\ \hline
1&$-2$&$4$ & $-2$ & $0$ \\ \hline
2&$1$&$2$ & $-1$ & $0$ \\ \hline
3&$0$&$0$ & $0$ & $0$ \\ \hline
\end{tabular}}}.
 \end{center}

\end{prop}

\begin{prop} \label{prop-ca-2}
For a partition $\lap=(a,b,\epsilon)$ $(\epsilon \in \{0,1\})$,
write $z \seteq a-b$ and $b'\seteq b-\epsilon$. Then we have
\begin{align}\label{eq: 44 closed}
| \Phi^{(4,4)} (\lap)|  =
\dfrac{b'z+\eta_\epsilon(z,b')\cdot (b',z)+\xi_\epsilon(z, b')}{8},
\end{align}
where $(x,y)\cdot (b',z)= xb'+yz$, and  $\eta_\epsilon(z,b')$ and $\xi_\epsilon(z,b')$ are defined on the congruence classes of $z$ and $b'$ by
\begin{center} {$\eta_0(z,b)=$
{\scriptsize \begin{tabular}{|c||c|c|}
\hline
$z\backslash b$ & 0&1\\ \hline \hline
0& $(2,2)$ &$(0,1)$   \\ \hline
1 & $(1,2)$ & $(1,1)$ \\ \hline
\end{tabular} }, \qquad  \phantom{and}  \qquad
  $\phantom{L}\eta_1(z,b')=$
{\scriptsize \begin{tabular}{|c||c|c|}
\hline
$z\backslash b'$ & 0&1\\ \hline \hline
0& $(0,0)$ &$(2,1)$   \\ \hline
1 & $(1,0)$ & $(1,1)$ \\ \hline
\end{tabular} },
}
 \end{center}
\begin{center} {\phantom{L} $\xi_0(z,b)=$
{\scriptsize \begin{tabular}{|c||c|c|c|c|}
\hline
$z\backslash b$ & 0&1&2&3\\ \hline \hline
0& $8$ & $0$ & $4$  &$0$ \\ \hline
1 & $6$ & $-3$ & $2$ &$1$\\ \hline
2 & $4$ & $-4$ & $4$ &$0$ \\ \hline
3 & $2$ & $1$ & $2$ &$1$\\ \hline
\end{tabular}}   \qquad  and  \qquad
$ \xi_1(z,b')=$
{\scriptsize \begin{tabular}{|c||c|c|c|c|}
\hline
$z\backslash b'$ & 0&1&2&3\\ \hline \hline
0& $0$ & $6$ & $0$  & $2$ \\ \hline
1& $0$ & $5$ & $-4$  & $1$ \\ \hline
2& $0$ & $2$ & $-4$  & $2$ \\ \hline
3& $0$ & $1$ & $0$  & $1$ \\ \hline
\end{tabular} }
}.
\end{center}
\end{prop}

\begin{cor} \label{cor-fi}
For each partition $\lap=(a,b,\epsilon)$ $(\epsilon \in \{0,1\})$, write $z \seteq a-b$ and $b'\seteq b-\epsilon$. Then we have
$$\omega_\epsilon(z,b') \seteq 2|\Phi^{(4,4)}(\lap)|   -|\Phi^{(2,4)}(\lap)|
 = \dfrac{\alpha_\epsilon (z,b') \cdot (b',z)+\beta_\epsilon(z,b')}{4},$$
where $(x,y)\cdot (b',z)= xb'+yz$, and  $\alpha_\epsilon(z,b')$ and $\beta_\epsilon(z,b')$ are defined on the congruence classes of $z$ and $b'$ by
\begin{center} {\phantom{Ll} $\alpha_0(z,b)=$
{\scriptsize \begin{tabular}{|c||c|c|}
\hline
$z\backslash b$ & 0&1\\ \hline \hline
0& $(1,1)$ &$(-1,0)$   \\ \hline
1 & $(0,1)$ & $(0,0)$ \\ \hline
\end{tabular}}, \qquad \phantom{and}
\qquad $\phantom{LL} \alpha_1(z,b') =$
{\scriptsize \begin{tabular}{|c||c|c|}
\hline
$z\backslash b'$ & 0&1\\ \hline \hline
0& $(-1,-1)$ &$(1,0)$   \\ \hline
1 & $(0,-1)$ & $(0,0)$ \\ \hline
\end{tabular}},
}
\end{center}
\begin{center} {\phantom{LL}  $ \beta_0(z,b) =$
{\scriptsize \begin{tabular}{|c||c|c|c|c|}
\hline
$z\backslash b$ & 0&1&2&3\\ \hline \hline
0& $4$ & $1$ & $2$  &$-1$ \\ \hline
1& $3$ & $0$ & $-1$  &$0$ \\ \hline
2& $2$ & $-3$ & $0$  &$-1$ \\ \hline
3& $1$ & $0$ & $1$  &$0$ \\ \hline
\end{tabular}}
 \qquad and  \qquad $\beta_1(z,b')=$
{\scriptsize \begin{tabular}{|c||c|c|c|c|}
\hline
$z\backslash b'$ & 0&1&2&3\\ \hline \hline
0& $0$ & $3$ & $-2$  &$1$ \\ \hline
1& $1$ & $0$ & $-3$  &$0$ \\ \hline
2& $-2$ & $-1$ & $-4$  &$1$ \\ \hline
3& $-1$ & $0$ & $-1$  &$0$ \\ \hline
\end{tabular}}.
}
 \end{center}
\end{cor}

Now we state and prove the main result of this subsection.

\begin{theorem} \label{thm :g34} For a partition $\lap=(a+2k,b+2k,\epsilon+2k)$ $(k \in \mathbb Z_{\ge 0}, \epsilon \in \{0,1\})$, write $z \seteq a-b$ and $b'\seteq b-\epsilon$. Then  we have
\[  \m_{\lap}(H'_{3,4})=\dfrac{(z+1)(b'+1)+\kappa_\epsilon(z,b')}{4} \quad \text{ and } \quad \mathfrak m_{\lap}(H_{3,4})=  \dfrac{b'z+\eta_\epsilon(z,b')\cdot (b',z)+\xi_\epsilon(z,b')}{8} .\]
Furthermore, for any $m \in \Z_{\ge1}$, we have
\begin{align}\label{eq: g34 general}
 \frac 1 2 \left( \prod_{i=1}^m (1+x_i)(1+x_i^2)+  \prod_{i=1}^m (1+x_i)(1-x_i^2) \right)  =\sum_{\la \trianglelefteq (3^m)} \omega_\epsilon(z,b') \,  S^{\U(m)}_{\la}(\mathbf x),
\end{align}
where $z,b',\epsilon$ are determined by the transpose $\lap$ of $\la \trianglelefteq (3^m)$.
\end{theorem}

\begin{proof}
The first assertion follows from Propositions \ref{prop-mu}, \ref{prop-ca-1}, \ref{prop-ca-2} and Lemma \ref{lem-sh-z}.
Let $d\gamma, d \gamma_1, d\gamma_2$ be the probability Haar measures on $H_{3,4}, H'_{3,4} , \U(1)\le \U(3)$, respectively.  We use  the coset decompositions \eqref{H34} and \eqref{Hp34} and the computation \eqref{eqn-Jz} and Proposition \ref{prop-main} to obtain
\begin{align*}
 &\sum_{\la \trianglelefteq (3^m)}\m_{\lap}(H_{3,4})  S^{\U(m)}_{\la}(\mathbf x)=\int_{H_{3,4}} \prod_{i=1}^m  \det(I+x_i\gamma) d\gamma \\ & = \dfrac{1}{2} \int_{H'_{3,4}} \prod_{i=1}^m  \det(I+x_i\gamma) d\gamma_1+ \frac 1 4 \int_{J\U(1)}  \prod_{i=1}^m  \det(I+x_i\gamma) d\gamma_2
+  \dfrac{1}{4}  \int_{\pmb \zeta_4 J\U(1)} \prod_{i=1}^m \det(I+x_i \gamma) d\gamma_2 \allowdisplaybreaks\\
& = \dfrac{1}{2}  \sum_{\la \trianglelefteq (3^m)}\m_{\lap}(H'_{3,4})  S^{\U(m)}_{\la}(\mathbf x)  +  \dfrac{1}{4}   \prod_{i=1}^m (1+x_i)(1+x_i^2)+  \dfrac{1}{4}   \prod_{i=1}^m (1+x_i)(1-x_i^2).
\end{align*}

Hence, using Corollary \ref{cor-fi},
\begin{align*}
\frac 1 2 \left( \prod_{i=1}^m (1+x_i)(1+x_i^2) +  \prod_{i=1}^m (1+x_i)(1-x_i^2) \right) &=\sum_{\la \trianglelefteq (3^m)}\left ( 2 \m_{\lap}(H_{3,4}) - \m_{\lap}(H'_{3,4}) \right )  S^{\U(m)}_{\la}(\mathbf x)  \\&=\sum_{\la \trianglelefteq (3^m)} \omega_\epsilon(z,b') S^{\U(m)}_{\la}(\mathbf x). \qedhere
\end{align*}
\end{proof}

\begin{remark} \label{rmk: g34} (1) The left hand side of~\eqref{eq: g34 general} can be written as a simple combination of the monomial symmetric functions in $m$-variables:
$$ \frac 1 2 \left( \prod_{i=1}^m (1+x_i)(1+x_i^2)+  \prod_{i=1}^m (1+x_i)(1-x_i^2) \right)  =\sum_{  \substack{ \la' \trianglelefteq (m^3)  \\ \la'_2 \equiv_2 0  } } \mathsf{m}^{(m)}_{\la}(\mathbf x).$$
where $\la' = (\la'_1,\la'_2,\la'_3)$.

\noindent
(2) By replacing $x_i$ by $-x_i$ in \eqref{eq: g34 general}, we obtain
\begin{align*}
\frac 1 2 \left( \prod_{i=1}^m (1-x_i)(1+x_i^2) +  \prod_{i=1}^m (1-x_i)(1-x_i^2) \right) &=\sum_{\la \trianglelefteq (3^m)} (-1)^{|\la|} \omega_\epsilon(z,b') S^{\U(m)}_{\la}(\mathbf x).
\end{align*}

\end{remark}

\begin{example} \label{ex: g34 huge}
(1) Let us see the case $m = 7$ in Theorem~\ref{thm :g34}. Then we have
 \begin{align}\label{eq: sl34 ex7}
 \frac 1 2  \left( \prod_{i=1}^7 (1+x_i)(1+x_i^2)+  \prod_{i=1}^7 (1+x_i)(1-x_i^2) \right)  =\sum_{\la \trianglelefteq (3^7)} \omega_\epsilon(z,b') \,  S^{\U(7)}_{\la}(\mathbf x)
\end{align}

The function $S^{\U(7)}_{(3,2^2, 1)}(\mathbf x)$ appears in the right hand side of~\eqref{eq: sl34 ex7} with the coefficient $-2$ as $\la'=(4,3,1)$ and hence $z=1$ and $b'=2$.
As a polynomial itself, $S^{\U(7)}_{(3^2,2^2, 1^{1})}(\mathbf x)$ contains $1239$ monomial terms and $S^{\U(7)}_{(3,2^2, 1^{1})}(\mathbf 1)= 3402$.
 There are $36$ Schur functions with negative coefficients in the right hand side of~\eqref{eq: sl34 ex7} including $S^{\U(7)}_{(3,2^2,1)}(\mathbf x)$. However, we obtain
a polynomial of $8192$ monomial terms with all coefficient $1$ in the left hand side of~\eqref{eq: sl34 ex7}.

(2) Let us see  the case $m = 20$ in Theorem~\ref{thm :g34}:
 \begin{align}\label{eq: sl34 ex20}
 \frac 1 2 \left( \prod_{i=1}^{20} (1+x_i)(1+x_i^2)+   \prod_{i=1}^{20} (1+x_i)(1-x_i^2)  \right)=\sum_{\la \trianglelefteq (3^m)} \omega_\epsilon(z,b') \,  S^{\U(20)}_{\la}(\mathbf x)
\end{align}
Here $S^{\U(20)}_{(2^9, 1^{2})}(\mathbf x)$ appears in the right
hand side of~\eqref{eq: sl34 ex20} with the coefficient $-3$ as
$\la'=(11,9)$ and hence $z=2$ and $b'\equiv_4 1$. One can check that
$S^{\U(20)}_{(2^9, 1^{2})}(\mathbf 1)= 12342120700$. There are $590$
Schur functions with negative coefficients in the right hand side
of~\eqref{eq: sl34 ex20}, and the specialization of the right hand
side of~\eqref{eq: sl34 ex20} at $\mathbf 1$ is equal to $2^{39}$.
This shows some systematic, even miraculous, cancelations of
monomial terms in Schur functions involved in this example.
\end{example}

Combining Theorems \ref{thm :g3} and \ref{thm :g34}, we obtain the following identity.
\begin{cor}  \label{cor:fep}
For each partition $\lap=(a+2k,b+2k,\epsilon+2k)$ $(k \in \mathbb Z_{\ge 0}, \epsilon \in \{0,1\})$, write $z \seteq a-b$ and $b'\seteq b-\epsilon$.
Then we have
\begin{align*}
\prod_{i=1}^m (1+x_i)(1-x_i^2) &=
\sum_{  (\la'_1, \la'_2, \la'_3) \trianglelefteq (m^3)  } (-1)^{\delta(\la'_2 \equiv_2 1)}\mathsf{m}^{(m)}_{\la}(\mathbf x)
 =  \sum_{\la \trianglelefteq (3^m)} \widetilde{\omega}_\epsilon(z,b') \,  S^{\U(m)}_{\la}(\mathbf x),
\end{align*}
where
$$\widetilde{\omega}_\epsilon(z,b') \seteq  \dfrac{\alpha_\epsilon (z,b') \cdot (b',z)+\widetilde{\beta}_\epsilon(z,b')}{2}.$$
Here  $\widetilde{\beta}_\epsilon(z,b')$ are defined on the congruence classes of $z$ and $b'$ by
\begin{center} {  $ \widetilde{\beta}_0(z,b) =$
{\scriptsize \begin{tabular}{|c||c|c| }
\hline
$z\backslash b$ & 0&1\\ \hline \hline
0& $2$ & $-1$   \\ \hline
1& $1$ & $0$   \\ \hline
\end{tabular}}
 \qquad and  \qquad $\widetilde{\beta}_1(z,b')=$
{\scriptsize \begin{tabular}{|c||c|c| }
\hline
$z\backslash b'$ & 0&1 \\ \hline \hline
0& $-2$ & $1$   \\ \hline
1& $-1$ & $0$   \\ \hline
\end{tabular}}.
}
 \end{center}
\end{cor}

Similarly, combining Theorem \ref{thm :g3} and Corollary \ref{cor:fep}, we obtain another identity below.
\begin{cor}  \label{cor:fep2} For each partition $\lap=(a+2k,b+2k,\epsilon+2k)$ $(k \in \mathbb Z_{\ge 0}, \epsilon \in \{0,1\})$, write $z \seteq a-b$ and $b'\seteq b-\epsilon$.
Then we have
\begin{align*}
 \frac 1 2 \left( \prod_{i=1}^{m} (1+x_i)(1+x_i^2)-   \prod_{i=1}^{m} (1+x_i)(1-x_i^2)  \right) & =   \sum_{  \substack{ (\la'_1,\la'_2,\la'_3) \trianglelefteq (m^3)  \\ \la'_2 \not\equiv_2 1 } } \mathsf{m}^{(m)}_{\la}(\mathbf x)
 =  \sum_{\la \trianglelefteq (3^m)} \widehat{\omega}_\epsilon(z,b') \,  S^{\U(m)}_{\la}(\mathbf x),
\end{align*}
where
$$\widehat{\omega}_\epsilon(z,b') \seteq  = \dfrac{-\alpha_\epsilon (z,b') \cdot (b',z)+\widehat{\beta}_\epsilon(z,b')}{4}.$$
Here  $\widehat{\beta}_\epsilon(z,b')$ are defined on the congruence classes of $z$ and $b'$ by
\begin{center} {  $ \widehat{\beta}_0(z,b) =$
{\scriptsize \begin{tabular}{|c||c|c|c|c|}
\hline
$z\backslash b$ & 0&1&2&3\\ \hline \hline
0& $0$ & $3$ & $-2$  & $1$ \\ \hline
1& $1$ & $0$ & $-3$  & $0$ \\ \hline
2& $-2$ & $-1$ & $-4$  & $1$ \\ \hline
3& $-1$ & $0$ & $-1$  & $0$ \\ \hline
\end{tabular}} \qquad and   \qquad $\widehat{\beta}_1(z,b')=$
{\scriptsize \begin{tabular}{|c||c|c|c|c|}
\hline
$z\backslash b'$ & 0&1&2&3\\ \hline \hline
0& $4$ & $1$ & $2$  & $-1$ \\ \hline
1& $3$ & $0$ & $-1$  & $0$ \\ \hline
2& $2$ & $-3$ & $0$  & $-1$ \\ \hline
3& $1$ & $0$ & $1$  & $0$ \\ \hline
\end{tabular}}.
}
 \end{center}
\end{cor}

\begin{remark} \label{rem:fep3}
Replacing $x_i$ with $-x_i$, we obtain identities for
\begin{align*}
\prod_{i=1}^m (1-x_i)(1+x_i^2) \qquad \text{ and } \qquad
\prod_{i=1}^m (1-x_i)(1-x_i^2)\end{align*}
and hence identities for
\begin{align*}
 \frac{1}{2} \left( \prod_{i=1}^m (1-x_i)(1+x_i^2) \pm \prod_{i=1}^m (1-x_i)(1-x_i^2) \right).
\end{align*}
\end{remark}

\section{Proofs for the Cardinalities of $\Phi^{(2,4)}$ and $\Phi^{(4,4)}$} \label{sec: Cardinality}

In this section, we will prove the explicit closed-form formulas of  $\Phi^{(2,4)}$ and $\Phi^{(4,4)}$, which are presented in Propositions~\ref{prop-mu}
and \ref{prop-ca-1}, respectively.  Throughout this section, we will explain the reason why the functions $\kappa_\epsilon(z,b)$, $\eta_\epsilon(z,b)$ and $\xi_\epsilon(z,b)$
are defined with respect to congruence classes modulo $2$ or $4$. We start this section with an example.

\smallskip

Let us consider the partition $\nu' = (a,b,0)=(a,b)=(7,2)$ of length $2$.
Note that
\begin{itemize}
\item $z=a-b=5$ and hence $z+1=6$, $b+1=3$,
\item among the partitions $(a,b)$ with
\begin{align} \label{eq: smallest ex}
a-b=5  \quad \text{ and } \quad b \equiv_4 2,
\end{align}
the partition $(7,2)$ is the smallest one.
\end{itemize}
Thus the set $\Phi(7,2)$ can be displayed by the following
$(6 \times 3)$-array as in  ~\eqref{eq: array and dimension}:
\begin{equation} \label{eq: ex for explain}
{\scriptsize \boxed{
\begin{matrix}
(7,2) & (7,1) &(7,0) \\
(6,2) & (6,1) &(6,0) \\
(5,2) & (5,1) &(5,0) \\
(4,2) & (4,1) &(4,0) \\
(3,2) & (3,1) &(3,0) \\
(2,2) & (2,1) & (2,0)
\end{matrix} }^{ \; z }_{\; 2} \ = \  \boxed{
\begin{matrix}
(z+b,b) & (z+b,b-1) &(z+b,0) \\
(z+b-1,b) & (z+b-1,b-1) &(z+b-1,0) \\
\cdot   &  \cdot & \cdot   \\
\cdot   &  \cdot & \cdot   \\
(b+1,b) & (b+2,1) &(b+1,0) \\
(b,b) & (b+1,1) &  (b,0)
\end{matrix} }^{ \; z }_{ \; b}
}
\end{equation}

Motivated by~\eqref{eq: ex for explain}, for $(k,l) \in \Z^2$ and $z,b \in \Z_{\ge0}$, we define the  set $\boxed{(k,l)}^{\; z}_{\; b}$ of pairs of integers:
\begin{align}
\boxed{(k,l)}^{\; z}_{\; b} \seteq \{ (p,q) \in \Z^2 \ | \  k+b\ge p \ge k, \;  z+l\ge q \ge l \},
\end{align}
whose cardinality is $(z+1) \times (b+1)$.

In~\eqref{eq: ex for explain}, one can check that, for $\nu'=(7,2)$, we have
\begin{itemize}
\item $\Phi^{(2,4)}(\nu') =\{  (4,0), \; (3,1), \; (7,1), \; (6,2), \; (2,2) \}$,
\item $\Phi^{(4,4)}(\nu') =\{  (4,0), \; (6,2), \; (2,2) \}$,
\end{itemize}
and hence $\Phi^{(2,4)}(\nu') \setminus \Phi^{(4,4)}(\nu') = \{ (3,1), \; (7,1) \}$.
Here one can notice that $(p,q) \in \Phi^{(2,4)}(\nu') \setminus \Phi^{(4,4)}(\nu')$ are located in the second column of~\eqref{eq: ex for explain},
since
\begin{align} \label{eq: exclusion}
\text{$q \equiv_4 1$ and hence $p \equiv_4 3$, yielding $p-q \not\equiv_4 0$. }
\end{align}

Let us consider a partition $\mu'=(11,6)$ which is the second smallest case in the sense of~\eqref{eq: smallest ex}.
Then its corresponding $\Phi(11,6)$ can be represented by the following
$(6 \times (3+4))$-array as follows:
\begin{equation} \label{eq: ex for explain2}
{\scriptsize
\boxed{
\begin{matrix}
(11,6) & (11,5) &(11,4) \\
(10,6) & (10,5) &(10,4) \\
(9,6) & (9,5) &(9,4) \\
(8,6) & (8,5) &(8,4) \\
(7,6) & (7,5) &(3,4) \\
(6,6) & (6,5) &(6,4) \\
\end{matrix}}^{\; z}_{ \; 2}
\ \boxed{
\begin{matrix}
(11,3) &(11,2) & (11,1) &(11,0) \\
(10,3) &(10,2) & (10,1) &(10,0) \\
(9,3) &(9,2) & (9,1) &(9,0) \\
(8,3) &(8,2) & (8,1) &(8,0) \\
(7,3) &(7,2) & (7,1) &(7,0) \\
(6,3) &(6,2) & (6,1) &(6,0) \\
\end{matrix} }^{\; z}_{ \; 3}
 }
\end{equation}
Here the left (resp. right) part  of~\eqref{eq: ex for explain2} coincides with  $\boxed{ (b+4,4)  }^{ \; z} _{ \; 2}$  $\left( \text{resp. }\boxed{ (b+4,0)  }^{ \; z} _{ \; 3} \right)$.

Then one can see that (i) the left part of~\eqref{eq: ex for explain2} can be obtained from~\eqref{eq: ex for explain} by adding $(4,4)$ for each entry
and  (ii) the right part of~\eqref{eq: ex for explain2} consists of four columns. Since  the number of $(p,q)$'s satisfying $p+q\equiv_40$
in $\boxed{ (b+4,4)  }^{ \; z} _{ \; 2}$  is the same as $ |\Phi^{(2,4)}(\nu')|$ by (i), and each row in $\boxed{ (b+4,0)  }^{ \; z} _{ \; 3}$ contains
$(p,q)$ with $p+q \equiv_4 0$ exactly once by (ii),  we have
\begin{align*}
|\Phi^{(2,4)}(\mu')| &=  |\Phi^{(2,4)}(\nu')| + (z+1) = 5+6.
\end{align*}
For the same reason as in~\eqref{eq: exclusion}, the $(p,q)$'s with $p+q \equiv_4 0$ with $q \equiv_4 1,3$ cannot satisfy the condition $p-q \equiv_4 0$.  Thus we obtain
\begin{align*}
|\Phi^{(4,4)}(\mu')| &=  |\Phi^{(4,4)}(\nu')| + \dfrac{(z+1)}{2} = 3+3.
\end{align*}
Since
$$
\Phi^{(4,4)}(c+4,d+4) =  \left(\Phi^{(4,4)}(c,d)+(4,4) \right) \sqcup  \boxed{ (d+4,0)  }^{ \; z}_{ \; 3}
$$
and the number of $(p,q)$'s satisfying $p+q \equiv_4 0$ (resp. $p+q \equiv_4 0$ and $p-q \equiv_4 0$) in $\boxed{ (d+4,0)  }^{\; z}_{ \; 3}$ does not change for
any $(c,d)$ with $c-d=5$ and $d\equiv_4 2$, the closed-form formulas are written in ~\eqref{eq: 24 closed} and~\eqref{eq: 44 closed},
which are determined by $|\Phi^{(k,4)}(\nu')|$,  $|\Phi^{(k,4)}(\mu')| -|\Phi^{(k,4)}(\nu')|$ $(k=2,4)$ depending on
$z$ and on the congruent classes of $z$  and $b$ modulo $4$.

\medskip

Now we generalize the argument above. We begin with a definition.

\begin{definition} Fix $\epsilon \in \{0,1\}$, $\mathsf{b}' \in \{0,1,2,3\}$ and $z \in \Z_{\ge 0}$. We call the partition
$$\nu'(z,\mathsf{b}',\epsilon) \seteq (z+\mathsf{b}'+\epsilon,\mathsf{b}'+\epsilon,\epsilon)$$  the \emph{base} partition for the triple $(z,\mathsf{b}',\epsilon)$
\end{definition}

For any finite set $X$ of pairs of integers, we denote by $\phi^{(2,4)}(X)$ (resp. $\phi^{(4,4)}(X)$)  the number of $(p,q)$'s in $X$ satisfying $p+q \equiv_4 0$
(resp.  $p+q \equiv_4 0$ and $p-q \equiv_4 0$). In particular, we write
$$\phi^{(u,4)}(\nu') \seteq \phi^{(u,4)}( \Phi(\nu')) =  |\Phi^{(u,4)}(\nu')| \qquad  (u=2,4)$$
for a partition $\nu'$ of the form $(z+b'+\epsilon,b'+\epsilon,\epsilon)$ $(z,b' \in \Z_{\ge 0}, \; \epsilon \in \{0,1\})$. Then the following lemma is obvious:

\begin{lem} \label{lem: obvious} For any $\boxed{(k,l)}^{\; z}_{\; b}$ and $(m,n) \in \Z^2$, we have
$$   \phi^{(u,4)} \left(  \;  \boxed{(k,l)}^{\; z}_{\; b} \right) =  \phi^{(u,4)} \left(  \;  \boxed{(k,l)}^{\; z}_{\; b} + (4m,4n) \right)   \qquad (u=2,4).$$
\end{lem}

\begin{lem} \label{lem: z+1}
For any partition $(a,b,\epsilon)$ with $a-b=z \in \Z_{\ge 0}$ and $\epsilon \in\{0,1\}$, we have
$$ \phi^{(2,4)}(a+4,b+4,\epsilon) -   \phi^{(2,4)}(a,b,\epsilon)   = z+1.$$
\end{lem}

\begin{proof} Let $z =a-b$ and $b'=b-\epsilon$. Then we have
$$\Phi(a+4,b+4,\epsilon) = \boxed{(b+4,\epsilon)}^{\; z}_{ \; b'+4}, \quad
\Phi(a,b,\epsilon) = \boxed{(b,\epsilon)}^{\; z}_{ \; b'}$$ and hence
\begin{align}\label{eq: decomp}
\Phi(a+4,b+4,\epsilon) = \left( \Phi(a,b,\epsilon) + (4,4) \right) \; \ssqcup  \;  \boxed{(b,\epsilon)}^{ \; z}_{ \; 3}.
\end{align}
By Lemma~\ref{lem: obvious}, $\phi^{(u,4)}(a,b,\ep) = \phi^{(u,4)} \left(  \;  \boxed{(b,\epsilon)}^{\; z}_{ \; b'} +(4,4)\right) $ $(u=2,4)$ and hence
$$ \phi^{(u,4)}(a+4,b+4,\epsilon) -   \phi^{(2,4)}(a,b,\epsilon)   = \phi^{(u,4)} \left( \; \boxed{(b,\epsilon)}^{ \; z}_{ \; 3}  \right) \qquad (u=2,4).$$
Recall that $\boxed{(b,\epsilon)}^{ \; z}_{ \; 3} $ consists of four columns with height $z+1$. Thus we have $$\phi^{(2,4)} \left( \; \boxed{(b,\epsilon)}^{ \; z}_{ \; 3}  \right)=z+1,$$
because each row in $\boxed{(b,\epsilon)}^{ \; z}_{ \; 3} $ contains $(p,q)$ with $p+q\equiv_4 0$ exactly once.
\end{proof}

\begin{lem} \label{lem: kappa}
For any base partition $\nu' = \nu'(z,\mathsf{b}',\epsilon)$, we have
$$ \phi^{(2,4)}(\nu')   =    \dfrac{(\mathsf{b}'+1)(z+1)+\kappa_\epsilon(z,\mathsf{b}')}{4}.$$
\end{lem}

\begin{proof} (1) Note that, when $\mathsf{b}'=3$ (resp. $z \equiv_3 4$), the integer $\kappa_\ep(z,\mathsf{b}')=0$. The assertion for this case comes from the fact that
$$    \Phi(\nu') =  \boxed{ (\mathsf{b}'+\ep,\ep) }^{\; z}_{\; \mathsf{b}'}   $$
consists of $4$-columns with height $z+1$ (resp. $(\mathsf{b}'+1)$-columns with height $z+1=4k$ for some $k \in \Z_{\ge1}$), implying that each row in the set contains $(p,q)$
with $p+q \equiv_4 0$ exactly once
(resp. each column in the set  contains $(p,q)$'s
with $p+q \equiv_4 0$ exactly $k$-times).

\noindent
(2)  Now let us consider the case when $\mathsf{b}' = 2$, $z \equiv_4 1$ and $\ep=1$. Write $z = 4k+1$ for some $k \in \Z_{\ge0}$.
Then $ \boxed{ (\mathsf{b}'+\ep,\ep) }^{\; z}_{\; \mathsf{b}'}  =  \boxed{ (\mathsf{b}'+\ep,\ep) }^{\; z}_{\; 2} $ can be described as follows:
\begin{equation} \label{eq: bp2z1ep1}
{\scriptsize \boxed{
\begin{matrix}
(\mathsf{b}'+z+1,3) & (\mathsf{b}'+z+1,2) &(\mathsf{b}'+z+1,1) \\
(\mathsf{b}'+z,3) & (\mathsf{b}'+z,2) & \underline{(\mathsf{b}'+z,1)} \\
(\mathsf{b}'+z-1,3) & \underline{(\mathsf{b}'+z-1,2)} & (\mathsf{b}'+z-1,1) \\
\vdots&\vdots &  \vdots \\
\underline{(b'+3,3)} & (\mathsf{b}'+3,2) &(\mathsf{b}'+3,1) \\
(\mathsf{b}'+2,3) & (\mathsf{b}'+2,2) &(\mathsf{b}'+2,1) \\
(\mathsf{b}'+1,3) & (\mathsf{b}'+1,2) & \underline{(\mathsf{b}'+1,1)}  \\
\end{matrix} }^{ \; z }_{\; 2}
}
\end{equation}
Here the underlined pairs $\underline{(p,q)}$ of integers satisfy the condition that $p+q \equiv_4 0$. One sees that the number $t_i$ of underlined $(p,q)$'s in the column with $q=i$ $(i=0,1,2)$ is given by
\begin{align}   \label{eq: bp2z1ep1 cols}
t_0 = k+1, \quad  t_1 = k \quad \text{ and } \quad t_2 = k.
\end{align}
Thus we have
$$ \phi^{(2,4)}(\nu')  = 3k+1 = \dfrac{ 3 \times (4k+2)-2  }{4}$$
which implies $\kappa_1(1,2)=-2.$

\noindent
(3) For the remaining cases, one can prove the formula by a similar argument.
\end{proof}

\begin{proof}[Proof of Proposition~\ref{prop-ca-1}] By Lemma~\ref{lem: z+1} and  Lemma~\ref{lem: kappa}, for a partition $(a,b,\ep)$ of the form
$$(z+\mathsf{b}'+4k+\ep,\mathsf{b}'+4k+\ep,\ep) \qquad (z=a-b, \; k \in \Z_{\ge 0} \text{ and } \mathsf{b'} \in \{ 0,1,2,3\} ),$$
we have
\begin{align*}
\phi^{(2,4)}(a,b,\ep) & = \dfrac{(\mathsf{b}'+1)(z+1)+\kappa_\epsilon(z,\mathsf{b}')}{4} + k \times (z+1) \\
& = \dfrac{(b'+1)(z+1)+\kappa_\epsilon(z,b')}{4}  \qquad (\because b' = b-\ep),
\end{align*}
as we desired.
\end{proof}

\begin{lem} \label{lem: eta}
For any partition $(a,b,\epsilon)$ with $a-b=z \in \Z_{\ge 0}$, we have
$$ \phi^{(4,4)}(a+4,b+4,\epsilon) -   \phi^{(4,4)}(a,b,\epsilon)   =  \dfrac{z+\eta_\epsilon(z,b)_1  }{2},$$
where  $\eta_\epsilon(z,b)_1$ denotes the first component of $\eta_\epsilon(z,b)$.
\end{lem}

\begin{proof} By~\eqref{eq: decomp}, it suffices to show that
$$  \phi^{(4,4)} \left( \; \boxed{(b,\epsilon)}^{ \; z}_{ \; 3} \right) = \dfrac{z+\eta_\epsilon(z,b)_1  }{2},$$
where  $ \phi^{(2,4)}\left( \; \boxed{(b,\epsilon)}^{ \; z}_{ \; 3} \right) = z+1$ by Lemma~\ref{lem: z+1} and the set  $\boxed{(b,\epsilon)}^{ \; z}_{ \; 3}$ can be described as follows:
\begin{equation} \label{eq: b ep z 3}
{\scriptsize \boxed{
\begin{matrix}
(b+z,\ep+3)  &(b+z,\ep+2)  & (b+z,\ep+1)  & (b+z,\ep) \\
(b+z-1,\ep+3)  &(b+z-1,\ep+2)  & (b+z-1,\ep+1)  & (b+z-1,\ep) \\
(b+z-2,\ep+3)  &(b+z-2,\ep+2)  & (b+z-2,\ep+1)  & (b+z-2,\ep) \\
\vdots&\vdots &  \vdots & \vdots \\
(b+2,\ep+3)  &(b+2,\ep+2)  & (b+2,\ep+1)  & (b+2,\ep) \\
(b+1,\ep+3)  &(b+1,\ep+2)  & (b+1,\ep+1)  & (b+1,\ep) \\
(b,\ep+3)  &(b,\ep+2)  & (b,\ep+1)  & (b,\ep)
\end{matrix} }^{ \; z }_{\; 3}
}
\end{equation}
Then, for each adjacent two rows, there are exactly two pairs of integers $(p_i,q_i)$ $(i=1,2)$
satisfying $p_i+q_i\equiv_4 0$, where exactly one of them does not satisfy the condition $p-q\equiv_4 0$ and the another satisfies the condition, as in~\eqref{eq: exclusion}.

(1) For $z \equiv_2 1$, we have always $\eta_\ep(z,b)=1$ and hence $\dfrac{z+\eta_\epsilon(z,b)_1  }{2} = \dfrac{z+1}{2}$. Thus the assertion for this case follows from the fact that $z+1 \equiv_2 0$.

(2) Let us consider the  case when $z \equiv_2 0$, $b \equiv_2 1$ and $\ep=0$. Then it suffices to consider the bottom row in~\eqref{eq: b ep z 3}:
\begin{equation*}
{\scriptsize \boxed{
\begin{matrix}
(b,3)  &(b,2)  & (b,1)  & (b,0)
\end{matrix} }
}
\end{equation*}
Then exactly one of $(b,1)$ and $(b,3)$ satisfies the condition that $p+q \equiv_4 0$ but it does not satisfies the condition that $p-q \equiv_4 0$, since $q \equiv_4 1,3$.
Hence $$  \phi^{(4,4)} \left( \; \boxed{(b,0)}^{ \; z}_{ \; 3} \right) = \dfrac{z}{2}.$$

(3) The remaining $3$-cases can be proved by a similar argument.
\end{proof}

\begin{lem}  \label{lem: xi} For any base partition $\nu' = \nu'(z,\mathsf{b}',\epsilon)$, we have
$$  \phi^{(4,4)}(\nu')   = \dfrac{b'z+\eta_\epsilon(z,\mathsf{b}')\cdot (\mathsf{b}',z)+\xi_\epsilon(z, \mathsf{b}')}{8}.$$
\end{lem}

\begin{proof}
(1) As in Lemma~\ref{lem: kappa}, let us consider the case when $\mathsf{b}' = 2$, $z \equiv_4 1$ and $\ep=1$. Write $z = 4k+1$ for some $k \in \Z_{\ge0}$.
Then the underlined pairs $(p,q)$ of integers in columns with $q=1$ or $3$ in~\eqref{eq: bp2z1ep1} cannot satisfy the condition $p-q \equiv_4 0$,
while the ones in the column  with $q=2$ in~\eqref{eq: bp2z1ep1} satisfy the condition as $p\equiv_4 2$.
Thus $\phi^{(2,4)}(\nu') = k$ by~\eqref{eq: bp2z1ep1 cols}. Since
$$   k = \dfrac{2(4k+1)+(1,0)\cdot(2,4k+1) -4  }{8},$$
our assertion holds in this case.

(2) For the remaining cases, one can prove the formula by a similar argument.
\end{proof}

\begin{proof}[Proof of Proposition~\ref{prop-ca-2}] By Lemma~\ref{lem: eta} and  Lemma~\ref{lem: xi}, for a partition $(a,b,\ep)$ of the form
$$(z+\mathsf{b}'+4k+\ep,\mathsf{b}'+4k+\ep,\ep) \qquad (z=a-b, \; k \in \Z_{\ge 0} \text{ and } \mathsf{b'} \in \{ 0,1,2,3\} ),$$
we have
\begin{align*}
\phi^{(2,4)}(a,b,\ep) & = \dfrac{b'z+\eta_\epsilon(z,\mathsf{b}')\cdot (\mathsf{b}',z)+\xi_\epsilon(z, \mathsf{b}')}{8}+ k \times  \dfrac{z+\eta_\epsilon(z,b)_1  }{2}\\
& = \dfrac{b'z+\eta_\epsilon(z,b')\cdot (b',z)+\xi_\epsilon(z,b')}{8} \qquad (\because b' = b-\ep),
\end{align*}
as we desired.
\end{proof}

\begin{remark}
After the first version of this paper was posted on the arXiv, Ronald C. King informed us that identities \eqref{it: mainA}-\eqref{it: mainC} in \textbf{Application} can be extracted from the work of Yang--Wybourne, Lascoux--Pragacz and King--Wybourne--Yang \cite{YW,LP,KWY} that appeared in the mathematical physics literature in the 1980's. Their methods are purely combinatorial and do not have direct connections to auto-correlation functions.
\end{remark}

\newcommand{\etalchar}[1]{$^{#1}$}

\end{document}